\documentclass[11pt,a4paper,ralign]{amsart}
\usepackage{amssymb}
\usepackage[mathscr]{euscript}
\usepackage[all]{xy}

\theoremstyle{plain}
\newtheorem{theorem}{Theorem}[section]
\newtheorem{lemma}[theorem]{Lemma}

\newtheorem{corollary}[theorem]{Corollary}
\theoremstyle{definition}
\newtheorem{definition}[theorem]{Definition}
\newtheorem{remarks}[theorem]{Remarks}

\newtheorem{remark}[theorem]{Remark}

\def\<#1>{\langle\, #1\,\rangle}

\newcommand\restr[2]{{% we make the whole thing an ordinary symbol
  \left.\kern-\nulldelimiterspace % automatically resize the bar with \right
  #1 % the function
  \vphantom{\big|} % pretend it's a little taller at normal size
  \right|_{#2} % this is the delimiter
  }}
\newcommand{\R}{\mathbb{R}}
\newcommand{\C}{\mathbb{C}}
\newcommand{\Z}{\mathbb{Z}}
\newcommand{\N}{\mathbb{N}}

\newcommand{\B}{\mathscr{B}}
\newcommand{\A}{\mathscr{A}}

\newcommand{\Cl}{\mathfrak{Cl}}

\newcommand{\cnaught}{{C_0(G)}}

\newcommand{\luc}{\mathit{LUC}(G)}

\font\seis=cmr6
\def\CB{\mathscr{CB}}
\def\luc{{\seis{\mathscr{LUC}}}}

\def\ruc{{\seis{\mathscr{RUC}}}}
\def\wap{{\seis{\mathscr{WAP}}}}
\def\B{{\mathscr{B}}}

\def\ap{{\seis{\mathscr{AP}}}}
\def\lc{{\seis{\mathscr{LC}}}}

\begin{document}
\title[Interpolation Sets and Quotients of Function Spaces]%
      {Interpolation sets and
       the size of quotients of function spaces on a locally compact group}

\author[Filali and Galindo]{M. Filali \and  J. Galindo}
%\thanks{This work was started while the second listed author was
% visiting the Department of Mathematical Sciences of University of
% Oulu  in December 2009. Hospitality and support is gratefully acknowledged.}
\thanks{Research of  the second named  author  supported by the Spanish Ministry of Science (including FEDER funds), grant
MTM2008-04599/MTM and  Fundaci\'o Caixa Castell\'o-Bancaixa, grant
P1.1B2008-26.}


\address{\noindent Mahmoud Filali,
Department of Mathematical Sciences\\University of Oulu\\Oulu,
Finland. \hfill\break \noindent E-mail: {\tt mfilali@cc.oulu.fi}}
\address{\noindent Jorge Galindo, Instituto Universitario de Matem\'aticas y
Aplicaciones (IMAC)\\\\ Universidad Jaume I, E-12071, Cas\-tell\'on,
Spain. \hfill\break \noindent E-mail: {\tt jgalindo@mat.uji.es}}


\date{\today}
\maketitle
\begin{abstract}
We devise a fairly general method for estimating the size of
quotients between algebras of functions on a locally compact group.
This method is based on   the concept of interpolation sets and
unifies the approaches followed by many authors to obtain particular
cases.

We find in this way that  there is a linear
 isometric copy of $\ell_\infty(\kappa)$ in each of the following
quotient spaces:
\begin{itemize}
\item[--] $\wap_0(G)/C_0(G)$ whenever $G$ contains a subset $X$ that is an $E$-set  (see the definition in the paper)
and $\kappa=\kappa(X)$ is  the minimal number of
compact sets required to cover $X$. In particular,  $\kappa=\kappa(G)$ when $G$ is a $SIN$-group.
%\item[--]
%$\wap(G)/(\ap(G)\oplus C_0(G))$ is non-separable whenever $G$ contains a subset $X$ that is an $E$-set.
\item[--] $\wap(G)/\B(G)$,  when $G$ is any locally compact group and $\kappa=\kappa(Z(G))$ and $Z(G)$ is the centre of $G$, or when $G$ is either an $IN$-group or  a nilpotent group, and $\kappa=\kappa(G)$.
%\item[--]  $\wap(G)/(\ap(G)\oplus C_0(G)),$ where $G$ and $\kappa$ are as in the foregoing item.
\item[--] $\wap_0(G)/\B_0(G)$, when $G$ and $\kappa$ are as in the foregoing item.
\item[--]
$\CB(G)/\luc(G)$, when  $G$ is any locally compact group that is neither compact nor discrete and
 $\kappa=\kappa(G)$.
\end{itemize}
 \end{abstract}

\section{Introduction}
The main focus throughout the paper will be on $C^*$-algebras of functions on a  locally
compact group $G$ with identity $e$. If $\ell_\infty(G)$ denotes the $C^*$-algebra of bounded,
scalar-valued functions on $G$ with the  supremum norm, our concern will be with the following
subalgebras of $\ell_\infty(G)$: the algebra
$\CB(G)$ of continuous bounded functions, the algebra  $\luc(G)$ of
bounded right uniformly continuous functions, the  algebra $\wap(G)$
of weakly almost periodic functions, the Fourier-Stieltjes algebra
$B(G)$, the uniform closure of $B(G)$ denoted by $\B(G)$ and best known as the Eberlein algebra,
the algebra $\ap(G)$ of almost periodic functions,  and the
algebra $C_0(G)\oplus\mathbb C1$,  where $C_0(G)$ consists of the functions in $\CB(G)$ vanishing at infinity.

The spectra of these algebras $\A(G)$ define some of the best-known   semigroup
compactifications in the sense of \cite{BJM}. These are compact right (or left) topological semigroups
$G^\mathcal \A$ having a dense, continuous, homomorphic copy of $G$ contained in their topological centres (i.e, the map
$x\mapsto sx\quad  (x\mapsto xs):G^\mathcal A\to G^\mathcal A$ is continuous for each $s\in G$.) For instance,  the compactification
$G^{\luc}$ is the spectrum of $\luc(G)$, and is usually referred to
as the $\luc(G)$- or $\lc(G)$-compactification of $G$.
It is the largest
semigroup compactification in the sense that
 any other semigroup
compactification is a quotient of $G^{\luc}.$
 When  $G$ is discrete, $G^{\luc}$ and the Stone-\v Cech compactification $\beta G$ are the same.
The $\wap$-compactification $G^\wap$ is the spectrum of $\wap(G)$;
it is the largest semitopological semigroup compactification. The
Bohr or  $\ap$-compactifiaction is the spectrum of $\ap(G)$ and
is the largest topological (semi)group compactification.

The Banach duals of these $C^*$-algebras can also be made into Banach algebras with a
convolution type product extending in most cases that of the group algebra
$L^1(G)$.  We may recall that  $L^\infty(G)$ is the Banach dual of the group algebra $L^1(G)$ and consists of all
scalar-valued functions which are measurable and essentially bounded with respect
to the Haar measure; two functions are identified if they coincide on a locally null
set, and the norm is given by the essential supremum norm.
We may also recall that the product making $L^\infty(G)$ into a Banach algebra is the first (or the second) Arens product on the second dual space $L^1(G)^{**}$
of the group algebra, and that $\luc(G)^*$ may be seen as a quotient Banach algebra of $L^1(G)^{**}.$ For more details, see  for instance \cite{DL}.
These two Banach algebras have been studied extensively in
recent years.
Particular attention
has been given to properties related to Arens regularity of the
group algebra $L^1(G)$ and to the topological centres of $G^{\luc}$,
$\luc(G)^*$ and $L^1(G)^{**}$. For the latest, see \cite{BIP} and the references therein.

 The definitions of all these function algebras will be given in the next section. But for the moment the following diagram summarizes already the inclusion relationships known to hold among these algebras.
See \cite[page 143]{chou82} for the first inclusion;
\cite[Lemma 2.1]{chou82} for the first equality;
\cite{deL1} or \cite[Theorem 4.3.13]{BJM} for the second equality;
 \cite[Corollary 4.4.11]{BJM} or \cite{Bu}   for
 the third inclusion;
 the rest is easy to check.
    \begin{align*}\cnaught\oplus\ap(G)&\subseteq\B(G)=\ap(G)\oplus \B_0(G)\subseteq
      \wap(G)=\ap(G)\oplus \wap_0(G)\\&\subseteq
      \luc(G)\cap\ruc(G)\subseteq \luc(G)\subseteq \CB(G)\\&\subseteq L^\infty(G).
    \end{align*}

When $G$ is finite, the diagram is trivial. When $G$ is infinite and
compact, the diagram reduces to $\CB(G)\subseteq L^\infty(G)$.

%If $\mu$ denotes the unique invariant mean on $\wap(G)$ (see \cite{BJM}, or \cite{Bu}), we put  \begin{align*}\wap_0(G)&=\{f\in \wap(G):\mu(|f|)=0\}\;\text{and}\\
%\B_0(G)&=\{f\in \B(G):\mu(|f|)=0\}.\end{align*}
%
The task of comparing these algebras and estimating the sizes of the quotients formed among them  has already been taken by many authors. We now give a brief review of what is known in this respect.

In the review below as well as in our study of quotients between the above algebras the compact covering number will appear at several points. We recall that the \emph{compact covering number} of a topological space $X$ is the smallest cardinal number $\kappa(X)$ of compact subsets of $X$ required to cover $X$.

\medskip

{\sc Comparing $L^\infty(G)$ with its subspaces}. Already in 1961, Civin and Yood proved in their seminal paper \cite{CY} that
the quotient space $L^\infty(G)/\CB(G)$ is infinite-dimensional for
 any non-discrete locally compact Abelian group  and deduced that the radical
of the Banach algebra $L^\infty(G)^*$ (with one of the Arens products as a product) is also infinite-dimensional.

This idea was  pushed further by Gulick in \cite[Lemma 5.2]{G} when $G$ is Abelian, and  proved that the quotient  $L^\infty(G)/\CB(G)$ is even non-separable
and so is the radical of $L^\infty(G)^*.$ Then Granirer proved in \cite{Gr} the same results for any non-discrete locally compact group.

A decade later, Young produced, for any infinite locally compact group $G$, a function in $L^\infty(G)$ which is not
in $\wap(G)$, proving the non-Arens regularity of the group algebra $L^1(G)$ for any such a group, see \cite{Y}.

There was also \cite[Theorem 4.2]{BF1}  where  the quotient $\luc(G)/\wap(G)$
was seen to contain  a linear isometric copy of $\ell_\infty(\kappa),$ where
$\kappa$ is the compact covering of $G$.
%(see below for definition).
A fortiori, the quotient
$L^\infty(G)/\wap(G)$ contains the same copy, a fact that was used in \cite[Theorem 4.4]{BF1} to deduce
 that the group algebra is even extremely non-Arens regular
in the sense of Granirer, whenever $\kappa $ is larger than or equal to the minimal cardinal $w(G)$ of  a basis of neighbourhoods at the identity.

It was also proved  in  \cite[Section 4]{BF1}
that  the quotient
$L^\infty(G)/\CB(G)$ always contains a linear isometric copy of
$\ell_\infty$, yielding extreme non-Arens regularity for the group algebra of compact metrizable groups.
Due to  a result by Rosenthal proved in \cite[Proposition 4.7, Theorem 4.8]{rose70}, larger copies of $\ell_\infty$ cannot be expected in $L^\infty(G)$ when $G$ is compact.
The question on extreme non-Arens regularity of the group algebra was recently settled by   the authors of the present paper using a technique inspired by Theorem \ref{main:constr}. We actually find  in \cite{FGenar}
that, for any compact group $G$, $L^\infty(G)/\CB(G)$  contains  a copy of $L^\infty(G)$.
This fact together with \cite[Theorem 4.4]{BF1}
gives
that $L^1(G)$ is extremely non-Arens regular for any infinite locally compact group.
\medskip

{\sc Comparing $\CB(G)$ with its subspaces}. In 1966, Comfort and Ross \cite[Theorem 4.1]{CR} compared the spaces $\CB(G)$ and $\ap(G)$ for an arbitrary topological group, and proved that they are equal if and only if $G$ is pseudocompact (i.e., every continuous scalar-valued function on $G$ is bounded).
In 1970, Burckel showed  in \cite{Bu} that $\CB(G)$ and $\wap(G)$ are equal if and only if
$G$ is compact.
In \cite{BB}, Baker and Butcher compared $\CB(G)$ and $\luc(G)$
for  locally compact  groups, and proved that these two spaces are equal
if and only if $G$ is either discrete or compact.
This result was extended recently by Filali and Vedenjuoksu in \cite[Theorem 4.3]{FV} to all topological groups which are not $P$-groups. The author deduced in \cite{FV} that if $G$ is a topological group which  is not a $P$-group,
 then $\CB(G)=\luc(G)$ if and only if $G$ is pseudocompact.
In \cite{Dz}, Dzinotyiweyi showed that the quotient $\CB(G)/\luc(G)$ is non-separable if $G$ is a non-compact, non-discrete, locally compact group.
This theorem was generalized in \cite[Theorem 3.1]{BF1} and \cite[Theorem 4.1]{BF2}, where
$\CB(G)/\luc(G)$) was seen to contain in fact a linear isometric copy of $\ell_\infty$ whenever $G$ is a non-precompact topological group
which is not a $P$-group. So this theorem improved actually also Dzinotyiweyi's result for locally compact groups.
For non-discrete, $P$-groups, the quotient $\CB(G)/\luc(G)$ was seen to be trivial in the case when for instance $G$
is a Lindel\"of $P$-group (see \cite[Theorem 5.1]{FV}), but may also contain a linear isometric copy of $\ell_\infty$
for some other $P$-groups (see \cite[Theorem 3.3]{BF1}).
In \cite[Theorem 3.1]{BF2}, using a technique due  to Alas (see \cite{OTA}), the quotient space $\CB(G)/\luc(G)$ was also seen to contain a linear isometric copy of $\ell_\infty$
whenever $G$ is a non-$SIN$ topological group.

In the locally compact situation, our answer in the present paper is  precise and definite. We prove, in Section 5, that there is a linear isometric copy of $\ell_\infty(\kappa)$ in $\CB(G)/\luc(G)$,
where as before $\kappa$ is the compact covering $G,$
if and only if $G$ is a neither compact nor discrete.
This leads again to a linear isometric copy
of $\ell_\infty(\kappa)$ into the quotient $L^\infty(G)/ \wap(G)$,
 and of course may be used to deduce again the  extreme non-Arens regularity of  of $L^1(G)$ when $\kappa(G)\ge w(G)\ge\omega$  as in
 \cite[Theorem 4.4]{BF1}.
\medskip

{\sc Comparing $\luc(G)$ with $\wap(G)$}. In 1972, Granirer showed that
$\luc(G)=\wap(G)$ if and only if G is compact \cite{G2}.

It is not difficult to check that $G^\luc$ is a semitopological semigroup (i.e., the  topological centre of $G^\luc$ is the whole of $G^\luc$)
if and only if $\luc(G)= \wap(G)$. The same observation can be
made also for $\luc(G)^*$. This means that
$G^\luc$ or $\luc(G)^*$ is a semitopological semigroup if and only if $G$ is compact, i.e., $G^\luc=G$ is a compact group and $\luc(G)^*$ coincides with the measure algebra $M(G)$.

 More recently, Granirer's result
was deduced by Lau and Pym in  \cite[Proposition 3.6]{LP}
as a corollary of their main theorem on the topological centre of
$G^\luc$ being $G$, and again by Lau and \"Ulger in \cite[Corollary
3.8]{LU} as a corollary of the topological centre of $L^1(G)^{**}$
being $L^1(G)$ \cite{LL}.

Moreover, Granirer showed  in the same paper that if $G$ is
non-compact and amenable, then the quotient $\luc(G)/\wap(G)$
contains  a linear isometric copy of $\ell_\infty$,
and so it is not separable.
 This result was extended by Chou in \cite{chou75} to $E$-groups (see below for definition),  then by Dzinotyiweyi in \cite{Dz} to all non-compact locally compact groups,
 and  generalized by Bouziad and Filali in \cite[Theorem 2.2]{BF1} to all non-precompact topological groups.
 Moreover, as already mentioned above,  this result was improved in \cite[Theorem 4.2]{BF1} when $G$ is a non-compact locally compact group,  by having a copy of $\ell_\infty(\kappa)$
     in the quotient $\luc(G)/\wap(G)$.
\medskip

 {\sc Comparing $\wap(G)$ with its subspaces}.
In the "regular" side of the inclusion diagram, when we compare $\wap(G)$ with
its subspaces, the situation is not simpler. It is true that the
Fourier-Stieltjes algebra $B(G)$ may be dense in $\wap(G)$ (i.e., $\wap(G)=\B(G)$), as in the case
of  minimally weakly almost periodic groups studied  by Veech, Chou and
Ruppert, see \cite{V}, \cite{chou80} and \cite{rup2}. For these groups,
  $\wap(G)=\ap(G)\oplus C_0(G)$.
     However, if $G$ is a non-compact group,   then   $B(G)$ is far
from being dense in $\wap(G)$ in general as it shall soon be explained.

When  comparing $\wap(G)$ with $\ap(G)$ and $C_0(G)$,
  we may recall first that $\wap(G)=\ap(G)\oplus \wap_0(G)$.
  Burckel proved in \cite{Bu} that $C_0(G)\subsetneq
\wap_0(G)$ when $G$ is an Abelian, non-compact, locally compact
group. In \cite{chou75}, Chou considered $E$-groups and proved that the
quotient $\wap_0(G)/C_0(G)$ contains a linear isometric copy of
$\ell_\infty$.
 In Section \ref{wapapc0}, we  improve
this result   by showing that $\ell_\infty$ may be
replaced by an isometric copy of the larger space
$\ell_\infty(\kappa(E))$ in  each of the quotient space $\wap_0(G)/C_0(G)$,
 where $\kappa(E)$ is the compact covering of the
$E$-set contained in $G$. So when $G$ is an $SIN$-group, these quotients contain a copy of $\ell_\infty(\kappa(G)).$
For the same class of groups, we prove also that the quotient $\wap(G)/(\ap(G)\oplus C_0(G))$ is non-separable.

  In Section 5, we deal with the non-compact, $IN$-groups,  and with non-compact nilpotent groups. In this class of groups, the results of the previous section shall be considerably improved.
     Rudin proved in \cite{ru} that $\B(G)\subsetneq
\wap(G)$ if $G$ is  a locally compact Abelian group and contains a
closed discrete subgroup which is not of bounded order. This was
followed by \cite{ra}, where Ramirez extended Rudin's result to any non-compact,
locally compact, Abelian group. Then in \cite{chou82}, Chou extended and
strengthened the theorem to all non-compact $IN$-groups and nilpotent groups by
showing that the quotient $\wap(G)/\B(G)$ contains a linear
isometric copy of $\ell_\infty$.

 We shall strengthen Chou's result in Section 5 by showing that, in these cases,
there is in fact a linear isometric copy of $\ell_\infty(\kappa)$ in the quotient spaces
$\wap(G)/\B(G),$ $\wap(G)/(\ap(G)\oplus C_0(G))$ and $\wap_0(G)/\B_0(G),$
where $\kappa$ is as before the compact covering of $G.$  Our method of proof also shows that $\wap(G)/\B(G)$ always contains a copy of $\ell_\infty(\kappa(Z(G)))$.

It is worthwhile to note that all this confirms an observation made in \cite[page 216]{BJM}, and gives indeed an indication on the size and
complexity of the $\wap$-compactification $G^{\wap}$ and the Banach
algebra $\wap(G)^*$.

\medskip

{\sc Outline.}
The underlying structure in many of the proofs that estimate the
size of $\A_2(G)/\A_1(G)$ for $C^\ast$-subalgebras $\A_1(G)\subseteq \A_2(G)$ of
$\ell_\infty(G)$, depends on the existence of sets of interpolation
for $\A_2(G)$  that are not sets of  interpolation for $\A_1(G)$ (see for instance \cite{BF1}, \cite{BF2}, \cite{chou75}, \cite{chou82} or  \cite{Dz}). One of the main
objectives of the present paper is to make that structure emerge in
a clear fashion. A first, but essential, step towards this objective
is to work with the right concept of interpolation sets.
 We will use here the general concept of interpolation
 set introduced in \cite{FG} that
extends several related classical ones and
show  how to apply it in
this setting. The resulting interpolation sets are characterized in
\cite{FG} in term of topological group properties, thereby making them
easier to manipulate. We finally illustrate the scope of our
approach by studying some concrete cases. We shall in particular
study  under this light the following quotients: $\wap_0(G)$ by $C_0(G)$ and $\wap(G)$ by
$\ap(G)\oplus C_0(G)$ for $E$-groups,   $\wap(G)$ by $\B(G)$, $\wap(G)$ by
$\ap(G)\oplus C_0(G)$ and $\wap_0(G)$ by $\B_0(G)$
for $IN$-groups and nilpotent groups,  $\CB(G)$ by $\luc(G)$ for locally compact groups.

\subsection{The function algebras}
We start by recalling the definitions of the function algebras we are interested in, for more
details the reader is directed for example to \cite{BJM}.

Let $G$ be a topological group.
For each function $f$ defined on $G$, the left translate $f_s$ of
$f$ by $s\in G$ is defined on $G$ by $f_s(t)=f(st)$. For each $s\in
G$, the left translation operator $L_s\colon \ell_\infty(G)\to
 \ell_\infty(G)$ is defined as $L_s(f)=f_s$. The supremum  norm of an
element $f\in \ell_\infty(G)$ will be denoted as $\|f\|_\infty$.

A function $f\in \ell_\infty(G)$
is  {\it right uniformly continuous} when, if for every $\epsilon>0$,
there exists a neighbourhood $U$ of $e$ such that
\[
|f(s)-f(t)|<\epsilon\quad\text{whenever}\quad st^{-1}\in U.\]
The algebra of right uniformly continuous functions on $G $ is denoted by $\luc(G).$

 A function $f\in\CB(G)$ is {\it almost periodic} when the set of all its
left (equivalently, right) translates  is a relatively norm
compact subset in $\CB(G).$
The algebra of almost periodic functions on $G $ is denoted by $\ap(G).$

A function $f\in\CB(G)$ is {\it weakly almost periodic} when the set of all
its left (equivalently, right) translates makes  a relatively weakly
compact subset in $\CB(G).$
The algebra of weakly almost periodic functions  on $G $ is denoted by $\wap(G).$

The {\it Fourier-Stieltjes algebra} $B(G)$ is the linear span of the set of all continuous positive definite
functions on $G$. Equivalently, $B(G)$ is the space of
coefficients of unitary representations of $G$ when $G$ is locally compact.
 As the Fourier-Stieltjes algebra is not uniformly
closed we will work with the {\it Eberlein algebra} $\B(G)$, which is the uniform closure of $B(G),$
 in symbols
$\B(G)=\overline{B(G)}^{\|\cdot\|_\infty}$.

Let $\mu$ be the unique invariant mean on $\wap(G)$ (see \cite{BJM}, or \cite{Bu}). As stated above,  put \begin{align*}\wap_0(G)&=\{f\in \wap(G):\mu(|f|)=0\},\\
\B_0(G)&=\{f\in \B(G):\mu(|f|)=0\}.
\end{align*}
In \cite[page 143]{chou82}, Chou denoted $B(G)\cap \wap_0(G)$ by $B_c(G)$, and observed that $\B_0(G) =\overline{B_c(G)}$ when $G$ is locally compact.

\subsection{The spectrum as a compactification} Let $G$ be a topological group,  $\A(G)\subseteq \ell_\infty(G)$ be  a unital  $C^\ast$-subalgebra and denote by $G^\A$ the
the spectrum (the set of non-zero multiplicative linear functionals) of $\A(G)$. Equipped with the topology of pointwise convergence,  $G^\A$ becomes a compact Hausdorff topological space. There is a canonical morphism  $\epsilon_\A \colon G\to G^\A$  given by evaluations
\[\epsilon_\A(s)(f)=f(s),\; \text{for every}\;f\in \A(G)\; \text{and}\; s\in G.\] This map
is continuous if and only if $\A(G)\subseteq \CB(G)$, and injective on $G$ if and only if $\A(G)$ separates the
points of $G$. We may recall, for example, that the map $\epsilon_\A$ is injective on $G$
(and in fact a homeomorphism  onto its image in $G^\A$ ) whenever $C_0(G)\subseteq\A(G).$ This is not
a necessary condition since it may also happen that $\epsilon_\A$ is injective when $C_0(G)\cap \A(G)=\{0\}$ as it is the case when  $G$ is a locally compact, maximally almost periodic and $\A(G)=\ap(G)$.
It may also happen that $\epsilon_\A$ is injective on  a given subset $T$ of $G.$ We will then identify
$T$ as a subset of $G^\A.$ This situation occurs when for example $T$ is an $\A(G)$-interpolation set.

The $C^*$-algebra $\A(G)$ is {\it left translation invariant} when $f_s\in\A(G)$ for every $f\in \A(G)$ and $s\in G.$
When $\A(G)$ is left translation invariant, we may define for every
$x\in G^\A$ and $f\in\A(G),$ the function on $G$ by $xf(s)=x(f_s).$
When $1,$ $f_s$ and $xf$ are  in $\A(G)$ for every $s\in G,$ $x\in G^\A$ and $f\in \A(G),$
we say that  $\A(G)$ is {\it admissible}.

When $\A(G)$ is an admissible $C^*$-subalgebra of $\CB(G)$, $G^\A$ can be  equipped with the product
$G^\A$  given by
%\begin{align*}
\[xy(f)=x(yf)\quad \text{for every}\quad x,y\in G^\A\;\text{and}\; f\in\A(G).\] $G^\A$ then becomes a
semigroup compactification of the topological group $G$ in the sense of \cite{BJM}. This means
that $G^\A$ is a compact semigroup having a continuous, dense, homomorphic, image
of $G$ such that the mappings
\[
x\mapsto xy\colon  G^\A\rightarrow  G^\A \,\,\text{ and } \,\,
x\mapsto \epsilon_\A(s)x\colon G^\A\rightarrow  G^\A
\]
are continuous for every $y\in G^\A$ and $s\in G$.

The algebras  $C_0(G)\oplus\C$,  $\ap(G)$, $C_0(G)\oplus \ap(G)$, $\B(G)$,  $\wap_0(G)\oplus\C$,  $\wap(G)$
and $\luc(G)$ are all known to be admissible, see for example \cite{BJM}. But when $G$  is locally compact, $\CB(G)$ is not admissible unless $G$ is either discrete or compact, see \cite{BB} or \cite{FV} for more.

When $G$ is a locally compact group and $\A$ is an admissible  $C^*$-subalgebra of $\luc(G)$, the semigroup compactification $G^\A$ has the {\it the joint continuity property}, that is,
the map \[ (s,x)\mapsto \epsilon_\A(s)x\colon G\times G^\A\rightarrow  G^\A\] is continuous.

A recent account on semigroup compactifications is given in \cite{gali10}.

\subsection{A few words on notation}
All our groups will be  multiplicative and their identity element will be denoted as $e$.
The characteristic function of a set $T$ will be denoted as $1_T$.
If $X$ is a set and $T\subseteq X$, given $f\in \ell_\infty(X)$, we define
$\|f\|_T=\sup\{|f(x)|\colon x\in T\}$ so that $\|f\|_\infty=\|f\|_X$.
The morphism  $\epsilon_\A$ maps  $G$ into $G^\A$ faithfully if $\A(G)$ separates points.  If $X\subseteq G$, we will denote
the closure  of  $\epsilon_\A(X)$ simply as $\overline X^{\A}$,  while  the closure of $X$ in $G$ will be
denoted as  $\overline X.$
The reason for this is that in most of our applications the algebra $\A(G)$ separates points  of $G$ and therefore $\epsilon_\A$ may be used to identify $G$ with a subset of $G^\A$.

A standard application of Gelfand duality identifies $\A(G)$ with $\CB(G^\A)$. Under this identification,
 to  every $f\in \A(G)$ there corresponds $f^\A\in \CB(G^\A)$ in such a way that the following diagram commutes
  \begin{equation}
    \label{eq:com}\xymatrix{
G \ar[r]^{\epsilon_{_{\A}} }\ar[dr]_f & G^{\A} \ar[d]^{f^\A } \\
&  \C }
  \end{equation}
When $\epsilon_\A$ is injective $f^\A$ can be seen as an \emph{extension} of $f$ to $G^\A$.

\section{Interpolation sets and quotients of function spaces}

We begin our work by introducing in precise terms the sets we will be using, and
then we prove the impact they have  in measuring the size of our quotient spaces
$\A_2(G)/\A_1(G)$. This is achieved in Theorem \ref{main:constr}.

It is worthwhile to note that this theorem may also
be applied to obtain most (if not all) of the results concerning the quotient spaces of the various function algebras
mentioned in the introduction;
it is of course necessary at each time to construct the required interpolation sets.

Our final main results in this section and in the rest of the paper concern $C^*$-algebras of bounded functions on a locally compact group, but definitions and properties shall also be proved for
a general Hausdorff topological group whenever this makes sense.

\begin{definition} \label{approx} Let $G$ be a topological group and $\A(G)\subseteq \ell_\infty(G)$.
A subset $T\subseteq G$ is said to be
\begin{enumerate}
  \item an \emph{$\A(G)$-interpolation set} if every bounded function
$f\colon T\to \C$ can be extended to a function $\overline{f}\colon G\to
\C$ such that $\overline{f}\in \A(G)$.
\item an \emph{approximable $\A(G)$-interpolation set} if it is an
 $\A(G)$-interpolation set and for every
 neighbourhood $U$ of $e$, there are open neighbourhoods $
V_1,V_2$ of $e$ with $\overline{V_1}\subseteq V_2\subseteq U$ such
that, for each $T_1\subseteq T$  there is   $h\in \A(G)$ with
$h(V_1T_1)=\{1\}$ and $h(G\setminus (V_2T_1))=\{0\}$.
\end{enumerate}
\end{definition}

\begin{remark}
$\A(G)$-interpolation sets for some concrete algebras $\A(G)\subseteq
\ell_\infty(G)$  have been a frequent object of study, see
\cite{gali10} and  \cite{FG} for more details and references.
See also \cite{GH} for the most recent account on the subject.

Approximable interpolation sets appear in the early 70's as a crucial step in Drury's proof of the union theorem of Sidon sets, see \cite{drury70}. Other well-known interpolation sets are also approximable as for instance \emph{translation-finite sets} considered  by Ruppert in \cite{rup} (and called $R_W$-sets by Chou in \cite{chou90}) that turn to be the approximable $\wap(G)$-interpolation sets of discrete groups,  see \cite{FG} for more on this respect.

When $G$ is discrete, the definition of approximable $\A(G)$-interpolation set is much simpler. In that case
$T\subset G$ is an approximable $\A(G)$-interpolation set if and only if $T$ is an $\A(G)$-interpolation such that $1_T\in \A(G)$, or equivalently, if every function supported on $T$ is in $\A(G)$.

It
should however be reminded that approximable $\A(G)$-interpolation sets
do not  make sense for every $C^\ast$-subalgebra $\A(G)$ of $\ell_\infty(G)$.
For example,  no subset in a non-compact locally compact group
can be an approximable
$\ap(G)$-interpolation set, see \cite[Section 3 and Corollary 4.24]{FG}.
\end{remark}

\subsection{The quotients}
 The following lemma contains
some  elementary consequences of the definitions of
interpolation and approximable interpolation sets.
The identification of $\overline{T}^\A$ with the  Stone-\v Cech compactification of $T$ (with the discrete topology) allows us to use the  powerful property of extreme disconnectedness
of the latter compactification. As the reader will quickly notice this is the key in the arguments leading to the main results in this section.
The main results start with a generalization of a theorem proved by  Chou \cite{chou82} for $B(G)$ (Lemma \ref{choufs}) to  arbitrary $C^*$-subalgebras of $\ell_\infty(G)$ . Along with some rather  technical lemmas, this  provides us with the conditions stated in Theorem \ref{main:constr} and Corollary \ref{cor:main:constr} under which
the quotient $\A_2(G)/\A_1(G)$ ($\A_1(G)\subset \A_2(G)$ being admissible  $C^\ast$-subalgebras  of
$\CB(G)$)  contains a linear isometric copy of $\ell_\infty(\kappa)$
for some cardinal $\kappa.$

\begin{lemma}
  \label{lem:isdisc} Let $G$ be a topological group.
Let   $\A(G)$ be  a $C^\ast$-subalgebra of $\ell_\infty(G)$ with $1\in\A(G)$ and
$T\subseteq G$.
\begin{enumerate}
\item  $T$ is   an $\A(G)$-interpolation set if and only if $\epsilon_\A$ is injective on $T$ and there is a homeomorphism between $\overline{T}^{\A}$ and  $\beta
T_d$, the Stone-\v{C}ech-compactification of $T$ equipped with the discrete topology, that leaves
the points of $T$ fixed.
\item $T$ is   an $\A(G)$-interpolation set if and only if for every pair of subsets $T_1,T_2\subset T$, $T_1\cap T_2=\emptyset$ implies $\overline{T_1}^{\A}\cap \overline{T_2}^{\A}=\emptyset$.
\item If  $T$ is an  $\A(G)$-interpolation set and  $f\colon T\to \C$ is a bounded function, then $f$
 has  an extension
$\overline{f}\in \A(G)$ with $\|\overline{f}\|_\infty=\|f\|_{T}$.
\item If $T$ is an  approximable $\A(G)$-interpolation set, then for every bounded function $h\colon T\to \C$ and  every neighbourhood
$U$ of the identity, there is $f\in \A(G)$  such that
\[\restr{f}{T}=h,\quad  f(G\setminus UT)=\{0\}\quad\text{ and}\quad
\|f\|_\infty=\|h\|_T.\]
\end{enumerate}
\end{lemma}

\begin{proof}
First observe that   $\epsilon_\A$ is injective on every  $\A(G)$-interpolation set  $T$: if $t_1\neq t_2\in T$, there is $f\in \ell_\infty(T)$ with $f(t_1)\neq f(t_2)$. Take $\bar{f}\in \A(G)$ extending $f$.
By \eqref{eq:com} $\bar{f}=\bar{f}^\A\circ \epsilon_\A$, hence $\epsilon_A(t_1)\neq \epsilon_\A(t_2)$.

Assertion (i) follows then from the universal property defining the Stone-\v Cech compactification
of a discrete space.
In fact, the restriction of the evaluation map $\epsilon_{\A}$ to $T$
gives a homeomorphism of the discrete set $ T_d$ onto its image in $G^\A.$  So $\overline{T}^{\A}$ is a (topological) compactification of $T_d$,
and  we may apply \cite[Corollary 3.6.3]{enge77}.

Assertion (ii) follows also directly from a well-known characterization of the Stone-\v{C}ech compactification
of a discrete space, see for instance
\cite[Corollary 3.6.2]{enge77}.

To prove (iii), let $f\colon T\to \C$ with $\|f\|_{T}=M$ be given. If
$B_M$ is the closed disc of radius $M$ centered at 0 (in $\C$), we
can use (i) and the universal property of $\beta  T_d$ to find a
continuous function $f^\beta\colon \overline{T}^\A\to B_M$ with
$\restr{f^\beta}{T}=f$. Then, by Tietze's extension theorem,  $f^\beta$ can be
extended to a continuous function $f^\A\colon G^\A\to B_M$, the
restriction $\restr{f^\A}{G}$ is then the desired extension.

To prove  (iv), let
$T$ be  an approximable $\A(G)$-interpolation  set. First, we find, using (iii),
$f_1\in \A$ with $\restr{f_1}{T}=h$ and
$\|f_1\|_\infty=\|h\|_T$. The definition of approximable
$\A(G)$-interpolation sets provides two  neighbourhoods $V_1, V_2$ with
$\overline{V_1}\subseteq V_2\subseteq U$ and $f_2\in \A$ such
that \[f_2(V_1T)=\{1\}\quad\text{ and }\quad f_2(G\setminus V_2T)=\{0\}.\]
Using \cite[3.2.20]{enge77}, we
can assume (taking the minimum of $f_2$ and the function that is
constant and equal to 1) that $\|f_2\|_\infty=1$. The product
$f_1\cdot f_2$ then coincides  with $h$ on $T$ and vanishes off
$V_2T$.
\end{proof}

\begin{remark} Note that if in the lemma above  $\A(G)\subseteq \CB(G)$, then $T$ is necessarily discrete since
\emph{every} bounded function on $T$ must be continuous.

Observe as well that the sole existence of an infinite $\A(G)$-interpolation set $T$ in $G$, implies that $G^\A$ contains a copy of $\beta T_d$, where $T_d$ is the discrete set $T$. The
compactification $G^\A$  is therefore  large and  topologically involved.
\end{remark}

The following theorem, due to Chou \cite{chou82},  has its roots in a
result of Ramirez (see Theorem 2.3 of \cite{dunklrami}) in the
Abelian setting. This theorem is  used by Chou, loc. cit., to find
an isometric copy of $\ell_\infty$ inside $\wap(G)/\B(G)$ for a
discrete group $G.$ This was originally
 the departing point of our paper.

\begin{theorem}(Chou, \cite[Lemma 3.11]{chou82})
  \label{choufs}
  Let $G$ be a discrete group. A subset
 $T\subseteq G$ fails to be a $B(G)$-interpolation
 set if and only if there is a bounded function $f\in \ell_\infty(G)$,
with  $\|f\|_\infty= 1$  such that \[f(G\setminus T)=\{0\}\quad\text{ and}\quad
 \|\phi-f\|_T\geq 1\quad\text{ for all}\quad \phi\in B(G).\]
\end{theorem}

\begin{remark}\label{b(g)=sidon}
  It is an immediate consequence of the previous theorem that
$\B(G)$-interpolation sets  are also $B(G)$-interpolation sets (i.e.,
Sidon sets). We do not know whether Theorem \ref{choufs} remains valid for all locally compact groups.
\end{remark}

The result in Theorem  \ref{choufs}  is  more natural when the function algebra is a
 $C^\ast$-subalgebra. It is not  surprising therefore that it holds for
 \emph{any}  $C^\ast$-subalgebra. Next lemma proves even more.

 \begin{lemma}\label{lem:nointer}  Let $G$  be
a topological group,  $\A_1(G)\subseteq \A_2(G)\subseteq
\ell_\infty(G)$ be two
%point-separating
$C^\ast$-subalgebras with $1\in\A_1(G)$, and let $(T_\eta )_{\eta<\kappa}$ be a family of disjoint  subsets of $G$  such that
\begin{enumerate}
\item each $T_\eta$ fails to be an $\A_1(G)$-interpolation set,
 \item $T=\bigcup_{\eta<\kappa}T_\eta $ is an approximable  $\A_2(G)$-interpolation set.
 \end{enumerate}
 Then for each  open neighbourhood
$U$ of $e$,
 there is  a function $f\in \A_2(G)$ with
$\|f\|_\infty=1$   such that \[ f(G\setminus
UT)=\{0\}\quad\text{ and}\quad \| f-\phi\|_{T_\eta}\geq 1\;\text{ for every $\eta<\kappa$ and every}\; \phi\in
\A_1(G).\]
\end{lemma}

\begin{proof}
%Obviously  $T$ cannot be an $\A_1(G)$-interpolation if a function $f$ as above exists.
 Let  $T=\bigcup_{\eta<\kappa}T_\eta $ be an approximable  $\A_2(G)$-interpolation set as stated in the lemma. Let $U$  be an open neighbourhood of $e$.

To avoid cumbersomeness, we abuse our notation and use the same letters to denote subsets of $T$ and their images in $G^{\A_1}$.
%This is an abuse of notation, but

Then, by  Statement (ii) of Lemma \ref{lem:isdisc},   each $T_\eta$ must  contain
two disjoint subsets $T_{1,\eta},T_{2,\eta}$ such that
 $\overline{T_{1,
 \eta}}^{\A_1}\cap \overline{T_{2,\eta}}^{\A_1}\neq \emptyset$.
  Define  for each $\eta<\kappa,$ a function $h_\eta\colon G\to [-1,1]$ supported on $T_\eta $ with \[h_\eta(T_{1,\eta})=\{1\}\quad \text{and} \quad h_\eta(T_{2,\eta})=\{-1\}.\]
Then consider the function  $h\colon G\to [-1,1]$ supported on $T$ and given by \[h(t)=h_\eta(t)\;\text{ if}\; t\in T_\eta\;\text{for some}\;\eta<\kappa.\]
By Statement (iv) of Lemma \ref{lem:isdisc}, there is a
 a function $f\in \A_2(G)$
 such that \[f(G\setminus  UT)=0,\;\;\restr{f}{T
 }=h\;\text{ and}\;\|f\|_\infty=\|h\|_T=1.\]
 Let now $\phi$ be any function in $\A_1(G),$ and take  $\varepsilon>0$. Given $\eta<\kappa$, we are going to prove that $\|f-\phi\|_{T_\eta}\geq 1-\varepsilon$.

  Take $p_\eta\in \overline{T_{1,\eta}}^{\A_1}\cap\overline{T_{2,\eta}}^{\A_1}$ and
pick  $t_{1,\eta}\in T_{1,\eta}$ and $t_{2,\eta}\in T_{2,\eta}$ with \[|\phi(t_{1,\eta})-\phi^{\A_1}(p_\eta)|<\varepsilon\quad\text{ and}\quad |\phi(t_{2,\eta})-\phi^{\A_1}(p_\eta)|<\varepsilon,\] where $\phi^{\A_1}$ denotes the extension of $\phi$ to $G^{\A_1}$.
 Then
\begin{align}\label{1}
  2&=|h_\eta(t_{1,\eta})-h_\eta(t_{2,\eta})| = |h(t_{1,\eta})-h(t_{2,\eta})|=|f(t_{1,\eta})-f(t_{2,\eta})|
   \\ \nonumber &\leq |f(t_{1,\eta})-\phi(t_{1,\eta})|
   \\ \nonumber &+|\phi(t_{1,\eta})-\phi^{\A_1}(p_\eta)|
  +|\phi^{\A_1}(p_\eta)-\phi(t_{2,\eta})|
  +|\phi(t_{2,\eta})-f(t_{2,\eta})|.
\end{align}
It follows that either $|f(t_{1,\eta})-\phi(t_{1,\eta})|\geq 1-\varepsilon$ or
$|f(t_{2,\eta})-\phi(t_{2,\eta})|\geq 1-\varepsilon$. Since $\varepsilon>0$ was arbitrary, we find that $\|f-\phi\|_{T_\eta}\geq 1$. Since  $\|f\|_\infty =1$ and $f(G\setminus UT)=\{0\}$, we see that $f$ is the required function.
\end{proof}

For the main theorem in this section, we need to recall the following definitions.
These sets are also essential for  the rest of the paper.

\begin{definition}
   Let $G$ be a topological group,  $T$ be a  subset of $G$ and $U$ be a neighbourhood of $e$.
We say that $T$ is \emph{right $U$-uniformly
discrete}  if \[Us\cap Us^\prime=\emptyset\quad\text{ for every}\quad s\neq
s^\prime\in T.\]
 The set $T$ being \emph{left $U$-uniformly discrete} is defined  analogously.
 % \[sU\cap s^\prime U=\emptyset\quad\text{  for every }\quad s\neq s^\prime\in T.\]
We say that $T$ is \emph{right uniformly discrete} (resp. \emph{left uniformly discrete}) when it is right $U$-uniformly discrete (resp. left $U$-uniformly discrete) for some neighbourhood $U$ of $e.$
If $T$ is both left and right uniformly  discrete, we say that $T$
is uniformly discrete.
 \end{definition}

\begin{lemma}
  \label{lem:unic}
  Let $G$ be  a locally compact group, $\A(G)$ be a $C^\ast$-subalgebra of $\CB(G)$,  $U$ be a compact neighbourhood of $e$, and  $T\subseteq G$
 be an approximable $\A(G)$-interpolation set that  can be  partitioned as
  $T=\bigcup_{\eta <\kappa}T_\eta$
   with $UT_\eta\cap UT_{\eta^\prime}=\emptyset$ whenever $\eta \neq \eta^\prime$.  Then there is a compact neighbourhood $V$ of the identity with $V^3\subseteq U$ such that given any two
   functions $f,g\in \ell_\infty(G)$ supported in $VT$ and a function $\mathbf{c}=(c_\eta)_{\eta<\kappa}\in  \ell_\infty(\kappa)$  such that
\begin{equation}
  \label{eq:cont}
\restr{f}{VT_\eta}=c(\eta) \restr{g}{VT_\eta}\quad\text{ for each}\quad \eta<\kappa,\end{equation}
 one has that:
\begin{enumerate}
  \item If $g$  is continuous, then so is $f$.
  \item If $T$ is right $U$-uniformly discrete, $\A(G)$ is an admissible $C^\ast$-subalgebra of $\luc(G)$,  then   $g\in \A(G),$ implies $f\in \A(G)$.
\end{enumerate}
\end{lemma}

\begin{proof}
First, consider the  two neighbourhoods  $V_1$ and $V_2$ provided by the definition of approximable $\A(G)$-interpolation sets for the neighbourhood $U$. We take $\overline{V_1}$ as $V$, and we can obviously assume that  $V^2\subseteq U$.

That $f$ is well defined follows from the relation $UT_\eta\cap UT_{\eta^\prime}=\emptyset$.

Let $g\in \CB(G)$ and $f\in \ell_\infty(G)$ be functions with
 $f(G\setminus VT)=g(G\setminus VT)=\{0\}$,      related as in \eqref{eq:cont}. We prove that $f$ is continuous considering separately continuity at interior points of $VT$ and points that do not belong to the interior of $VT$.

Let  $s\in G$ that is not an interior point of $VT$. Since $s$ can be approached from $G\setminus VT$, we see by continuity that $g(s)=0$.
By checking the cases when $s\in VT$ and when $s\notin VT$, we deduce that $f(s)=0$ as well.
Now let $(x_\alpha)$ be  a net in $G$ converging to $s$.
 We can assume that either $(x_\alpha)\subset VT$ or $(x_\alpha)\subset G\setminus VT$.
  In the former  case,  we use the fact that $|f(x_\alpha)|\leq \|\mathbf{c}\|_\infty\cdot |g(x_\alpha)|$ and conclude  that $\lim_\alpha f(x_\alpha)=0=f(s)$.
 In the other case when ($(x_\alpha)\subset G\setminus VT$), it is clear that $\lim_\alpha f(x_\alpha)=0=f(s)$. The continuity of $f$ at $s$ follows.

Suppose now that $s$ is an interior point of $VT$.  Pick $\eta<\kappa$ such that $s=vt$ with $v\in V$ and $t\in T_\eta$, and let $W$ be a neighbourhood of the identity with $Ws\subseteq VT$ with
$|g(s)-g(s')|<\epsilon$ for every $s^\prime\in Ws.$ Let $w\in W$ be such that $s'=wvt$  and notice that
 $s^\prime=ws=wvt=v_0t_0$ for some $v_0\in V$ and $t_0\in T$ implies that $t_0\in T_\eta$.
Therefore, \[|f(s)-f(s')|=|c(\eta)g(s)-c(\eta)g(s')|\quad\text{ for every}\quad s'\in Ws,\] and so the
continuity of $f$ at interior points follows as well.

We now assume $\A\subseteq \luc(G)$. Define a function $\varphi$ on $T$ by $\varphi(t)=c(\eta)$ for every $t\in T_\eta$. Since $T$ is an $\A(G)$-interpolation set, we may extend
$\varphi$ to a function $\overline{\varphi}\in \A(G).$ By Lemma \ref{lem:isdisc} (iv), we can assume that $\overline{\varphi}(G\setminus VT)=\{0\}$.
If $g^\A$ and $\varphi^\A$ denote the respective extensions of
$g$ and $\overline{\varphi}$ to $G^\A$, we  define $f^\ast\colon G^\A\to\C$ by
\begin{align}
  f^\ast(vp)&=  \varphi^\A (p)\cdot   g^\A(vp) \mbox{ if $v\in V$ and $p\in \overline{T}^\A$}\\ \nonumber
  f^\ast(x)&= 0 \mbox{ if $x\notin V\overline{T}^\A$.}
\end{align}
We check that $f^\ast$ is a well-defined, continuous extension of $f$  to $G^\A$.

\emph{(1) $f^\ast$ is well defined.} It might happen that some $vp\in V\overline{T}^\A$ admits two different decompositions. We  check that
the definition of $f^\ast$ does not depend of the choice of the decomposition. Suppose therefore that $v_1p_1=v_2p_2$ with $v_1,v_2\in V$ and $p_1,p_2\in \overline{T}^\A$.

If $p_1\neq p_2$,
we may  choose  $T_1, T_2\subset T$  such that
\[\overline{T_1}^\A\cap \overline{T_2}^\A=\emptyset,\; p_1\in \overline{T_1}^\A\quad\text{ and}\quad p_2 \in \overline{T_2}^\A\]
(this is possible by (ii) of Lemma \ref{lem:isdisc}).
%(simply choose $P$ a closed neighbourhood of $p$ with $p_1\notin P$ then take  $T_0=P\cap T$).
Since  $T$ is approximable, we  may pick   $h \in \A(G)$ such that \[h(VT_1)=\{1\}\quad\text{ and}\quad h(G\setminus V_2T_1)=\{0\}.\]

 Recalling  that multiplication by elements of $G$ is  continuous on $G^\A$, it is clear that
 $h^\A(v_1p_1)=1$.
By the same reason, if $(t_\alpha)$ is a net in $T_2$ converging to $p_2$,
we have that $v_2p_2=\lim_\alpha v_2 t_\alpha$. But since  $T$ is right $U$-uniformly discrete,   no element $v_2t_\alpha$ can be in $V_2T_1$, hence $h^\A(v_2p_2)$ must be zero.  This contradiction shows that \begin{equation}\label{Veech}v_1p_1=v_2p_2\;\text{ implies}\;p_2=p_1.\end{equation}
This shows already that $f^\ast$ is well defined, since the equalities $v_1p_1=v_2p_2$ and $p_1=p_2$ give us \[f^\ast(v_1p_1)=\varphi^\A(p_1)g^\A(v_1p_1)=f^\ast(v_2p_2).\]
(In fact, $v_1$ and $v$ must be also equal by the same argument, but this is enough for our purposes.)
\medskip

\emph{(2)
$f^\ast$ coincides with $f$ on $G$.}  Since  $V\overline{T}^\A\cap G=V(\overline{T}^\A\cap G)=VT$ and $f(G\setminus VT)=\{0\}$,  we readily see that $f$ and $f^\ast$ coincide on $G\setminus VT$.
 Let on the other hand $s=vt$ with $v\in V$ and $t\in T_\eta$.
Then
\[  f^\ast(s)=\varphi(t)g(s) =c(\eta) g(s)=f(s).\]
\medskip

\emph{(3) $f^\ast$ is continuous.}
%By Ellis-Lawson's Theorem (see for example \cite[Theorem 4.2]{BJM}), the map
%\[G\times G^\A\to G^\A: (s,x)\mapsto sx\] is jointly continuous.
%It is enough to show that $\lim_\alpha f(s_\alpha)=f^\ast(x)$ whenever $(s_\alpha)$ is a net in $G$ converging to $x$ in %$G^\A.$
Using the joint continuity property, we see that $V\overline T^\A$ is closed in $G^\A.$ So the continuity of $f^\ast$
at the points outside of  $V\overline T^\A$ is clear.

  We divide the case $x=vp\in V\overline T^\A$ into two subcases. Suppose first that $x$ is an interior point $V\overline T^\A$, and let $(q_\alpha)$ be a net in $G^\luc$
  converging to $x$. Then $(q_\alpha)$ is eventually of the form $(v_\alpha p_\alpha)$ with $(v_\alpha)$ in $V$ and $(p_\alpha)$ in
  $\overline T^\A.$  By taking subnets if necessary, we may assume that $\lim_\alpha v_\alpha=v_0$ in $V$ and $\lim_\alpha p_\alpha=p_0$ in $\overline T^\A.$
  Accordingly, $x=vp=v_0p_0,$ and applyig \ref{Veech}, we see that $p=p_0.$ Therefore,
  \[f^\ast(q_\alpha)=\varphi^\A(p_\alpha)g^\A(q_\alpha)\longrightarrow \varphi^\A(p_0)g^\A(x)=\varphi^\A(p)g^\A(x)=f^\ast(x),\]
  as required. The second subcase is when $x$ is outside the interior of $V\overline T^\A.$ Here, we may assume that the net $(q_\alpha)$ given to converge to $x$
  is also outside $V\overline T^\A$, and so  $g^\A(q_\alpha)=0$ for every $\alpha$. Since $g^\A$ is continuous, we deduce that
  $f^\ast(x)=\varphi^\A(p)g^\A(x)=0,$ as required.

 \medskip

From (1), (2), (3) we conclude that $f\in \A$.
\end{proof}

\begin{remarks}
(i) A known theorem due to Veech asserts that the left action of a locally compact group $G$ on $G^\luc$ is free, i.e., $gx\ne x$  for every $x\in G^\luc$ and $g\in G,$ $g\ne  e$, see \cite{Ve}, or \cite{P} for a shorter proof.
%This propery was studied by Baker and Filali in \cite{BF} and by Filali in \cite{F} in  $G^\wap$ by taking points in the closure of t-sets.
   The proof of the previous Lemma reveals that Veech's property in fact holds in $G^{\A}$ at any point in the closure of the approximable $\A(G)$-interpolation sets with $\A\subset \luc(G)$. That is, if $T$ is any such a set, $x\in\overline{ T}^\A$  and $g\ne  e$ in $G$, then $gx\ne x$ and $xg\ne x$ in $G^\A.$ This property was proved in $G^\wap$ in \cite{BF} and \cite{F3} using $t$-sets. $t$-Sets are by \cite{FG}
approximable $\wap(G)$-interpolation sets. We will return to these matters in a forthcoming work.
\vskip0.3cm

(ii) It could also be worth  to mention that for metrizable locally compact groups the condition on $T$ in (ii) is redundant. Indeed, by  \cite[Theorem 4.9]{FG},
every $\luc(G)$-interpolation subset of a metrizable group is right uniformly discrete.
\end{remarks}

%Several interesting corollaries can be readily derived  from Lemma \ref{lem:unic}.

%To see this, just take a neighbourhood $U$ of $e$ which does not contain $g$ and use the fact the property of  $T$
%to separate gp and p with a function from $\A.$

% In case of $\A=\wap(G)$, next result was proved by Ruppert in \cite[Theorem 7]{rup}.
%
%\begin{corollary}  Let $G$ be a discrete group. Then, $T$ is an approximable $\A(G)$-interpolation set
%if and only if any bounded function defined on $T$ with value $0$ outside $T$ is in $\A.$
%\end{corollary}
%
%\begin{proof} Sufficiency is clear. To see the converse,
%we use again the property of $T$ that the characteristic function $1_T$ is in $\A.$
%So taking  $h=1_T$ in the proof of Lemma \ref{lem:unic} shows in particular that the function defined by $f(t)=c(t)$ on $T$ and $0$ off $T$  is  in $\A$  for every bounded function $c$ on $T$.
%\end{proof}

\begin{theorem}\label{main:constr}
 Let $G$  be a locally compact  group and let
$\A_1(G)\subset \A_2(G)\subseteq \luc(G)$ be two unital
$C^\ast$-subalgebras of $\ell_\infty(G)$ with $\A_2(G)$ admissible. Let, in addition, $U$ be a compact   neighbourhood of the identity such that $T$ is right $U$-uniformly discrete.  Suppose that  $G$ contains a
family of sets
   $\{T_\eta\colon \eta<\kappa\}$ such that
   \begin{enumerate}
     \item $T_\eta\cap T_{\eta^\prime}=\emptyset$ for every
   $\eta\neq \eta^\prime<\kappa$,
   \item $T_\eta$ fails to
be an $\A_1(G)$-interpolation set for every $\eta<\kappa$, and
\item $T=\bigcup_{\eta<\kappa}T_\eta$ is an approximable $\A_2(G)$-interpolation
set.
   \end{enumerate}
Then there is  a linear isometry  $\Psi
\colon \ell_\infty(\kappa)\to \A_2(G)/\A_1(G)$.
\end{theorem}

\begin{proof}
Let $V$ be the  neighbourhood of the identity provided by Lemma \ref{lem:unic}.

Since $T=\bigcup_{\eta<\kappa}T_\eta$ is an approximable $\A_2(G)$-interpolation
set and each $T_\eta$ fails to be an $\A_1(G)$-interpolation set, we take from
Lemma \ref{lem:nointer}
a function $f\in \A_2(G)$ with
 $\|f\|_\infty=1$
such that
\[f(G\setminus VT)=\{0\}\quad\text{and}\quad
\|f-\phi\|_{T_\eta}\geq 1\quad\text{ for all}\quad \phi\in \A_1(G)\quad\text{and}\quad \eta<\kappa.\]
%Since each $T_\eta$ fails to be an $\A_1(G)$-interpolation set, we pick by Statement (ii) of \ref{lem:isdisc} two disjoint subsets $T_{\eta,1}$
%and $T_{\eta,2}$ of $T_\eta$ such that  $\cl_{G^\A_1}T_{\eta,1}\cap \cl_{G^\A_1}T_{\eta,2}\ne\emptyset$.
%Define for each $\eta<\kappa,$ a function $f_\eta$ on $T_\eta$
%by $f_\eta(T_{\eta,1})=\{1\}$ and $f_\eta(T_{\eta,2}=\{-1\})$.
For each
$\mathbf{c}=(c_\eta)_{\eta<\kappa}\in\ell_\infty(\kappa)$, we define  the
function $f_{\mathbf{c}}:G\to\C$  supported in $VT$ by
\[f_{\mathbf{c}}(vt)=c_\eta f(vt)\quad\text{ if}\quad t\in T_\eta\quad\text{and}\quad \eta<\kappa,\]
i.e., with  the notation of Lemma \ref{lem:unic}, $\restr{f_{\mathbf{c}}}{VT_\eta}=\mathbf{c}(\eta) \restr{f}{VT_\eta}$.

Then  $f_{\mathbf{c}}\in \A_2(G)$ by (ii) of Lemma \ref{lem:unic}. Obviously, the map
$\Psi \colon \ell_\infty (\kappa) \to \A_2(G)/\A_1(G)$ given by
\[\Psi(\mathbf{c})=f_{\mathbf{c}}+\A_1(G)\quad\text{for every}\quad \mathbf{c}\in\ell_\infty(\kappa)\] is linear. We next check that it is isometric.

 The same argument of \cite[Theorem 3.12]{chou82} shows now  that, for every $\eta_0 < \kappa,$
\begin{align*}
\|\Psi\left((c_\eta)_{\eta<\kappa}\right)\|_{\A_2(G)/\A_1(G)}&= \inf
\{\|f_{\mathbf{c}}-\phi\|_\infty:\phi\in
\A_1(G)\}\\& \ge \inf \{\|f_{\mathbf{c}}-\phi\|_{T_{\eta_0}}:\phi\in \A_1(G)\}
\\&
=\inf \{\|c_{\eta_0}f-\phi\|_{T_{\eta_0}}:\phi\in \A_1(G)\}
\\&
= |c_{\eta_0}|\inf \{\|f-\phi\|_{T_{\eta_0}}:\phi\in \A_1(G)\}
\\&
\ge |c_{\eta_0}|,
\end{align*}
where the last inequality follows from the choice of $f$.
Since, obviously, \[\|\Psi(\mathbf{c})\|_{\A_2(G)/\A_1(G)}\leq
\|f_{\mathbf{c}}\|_\infty=\|\mathbf{c}\|,\quad\text
{for every}\quad \mathbf{c}=(c_\eta)_{\eta<\kappa}\in
\ell_\infty(\kappa),\] we see
 that $\Psi$ is the required isometry.
\end{proof}

\begin{corollary}\label{cor:main:constr}
If in the above theorem $\A_2(G)=\CB(G)$ and $T$ is not assumed to be  right $U$-uniformly discrete but still $UT_\eta\cap UT_{\eta^\prime}=\emptyset$, then the quotient
$\CB(G)/\A_1(G)$ contains a linearly  isometric copy of $\ell_\infty(\kappa)$.
\end{corollary}

\begin{proof}
  The proof of Theorem \ref{main:constr} remains valid in this case applying (i) of Lemma \ref{lem:unic} instead of (ii).
\end{proof}

\begin{remark}
Two $C^\ast$-subalgebras of $\ell_\infty(G)$ may be different,  and yet
produce a small quotient (i.e., separable), for example if $G$ is a
minimally weakly almost periodic group (\cite{chou90}, \cite{ru},
\cite{V}) then $\wap(G)/\ap(G)=C_0(G)$. If $G=SL(2,\mathbb R),$ then
$\wap(G)=C_0(G)\oplus\mathbb C1,$ and so $\wap(G)/C_0(G)=\mathbb
C.$
In the theorem and corollary above, we have just met conditions under which this is not so.
\end{remark}

\begin{corollary} Under the hypotheses of Theorem \ref{main:constr}
or Corollary \ref{cor:main:constr}, the quotient space
$\A_2(G)/\A_1(G)$ is non-separable.
\end{corollary}

\section{Interpolation sets }

The definitions in this section gather the topological group-theoretic
properties that will correspond to  the interpolation sets needed in the three sections that follow.
Once these interpolation sets are at hand, an application of Theorem \ref{main:constr}
and Corollary \ref{cor:main:constr}
will lead immediately to the desired conclusion on the quotients.

In addition to the uniformly discrete sets defined in the previous sections we shall also need  the following sets.

\begin{definition}
Let $G$ be a non-compact topological group.
We say that a subset $S$ of $G$ is
\begin{enumerate}
\item
{\it right translation-finite} if every  infinite subset $L\subseteq G$
contains a finite subset $F$ such that $\bigcap\{b^{-1}T\colon b\in
F\}$ is finite; {\it left translation-finite} if every  infinite subset $L\subseteq G$
contains a finite subset $F$ such that $\bigcap\{Tb^{-1}\colon b\in
F\}$ is finite; and {\it translation-finite} when it is both right and left translation-finite.
\item
{\it right translation-compact} if every  non-relatively compact subset $L\subseteq G$
contains a finite subset $F$ such that $\bigcap\{b^{-1}S\colon b\in
F\}$ is relatively compact; {\it left translation-compact} if every  non-relatively compact subset $L\subseteq G$
contains a finite subset $F$ such that $\bigcap\{Sb^{-1}\colon b\in
F\}$ is relatively compact; and \emph{translation-compact} when it is both left and right translation-compact.
\item a \emph{right $t$-set} (\emph{left $t$-set}) if there exists a compact subset $K$ of $G$ containing $e$ such that $gS\cap S$ (respectively, $Sg\cap S$) is
relatively compact for every $g\notin K$; and a \emph{$t$-set} when it is both a  right and a left $t$-set.
 \end{enumerate}
\end{definition}

We  also need  to establish the range of locally compact groups to which our methods apply in the next two sections, these are those locally compact groups  for which the existence of  a good supply of $\wap$-functions is guaranteed.

Recall that a locally compact group $G$ is an {\it $IN-$group} if it has an invariant neighbourhood of the identity.
We recall also from \cite{chou75}, that a locally compact group $G$
%with identity $e$
is an {\it E-group} if it contains a non-relatively compact set $X$
such that for each neighbourhood $U$ of $e,$ the set \[\bigcap
\{x^{-1}Ux: x\in X\cup X^{-1}\}\]is again a neighbourhood of $e.$
The set $X$ is called an {\it E-set}. This is a large class of
locally compact groups. This    includes of course all non-compact
$SIN-$groups, the groups with a non-compact centre such as the
matrix group $GL(n,\mathbb{ R})$, and the direct product of any
$E$-group with any locally compact group.

 A detailed study  of  approximable $\luc$- and
 $\wap(G)$-interpolation sets,
 with some   precise characterizations,  is carried  out in the recent paper  \cite{FG}. We summarize in Lemma \ref{lem:lucint} the results that will be needed in the present paper.

 \begin{lemma}\label{lem:lucint} (\cite[Lemma 4.8 and Proposition 3.3 (iii)]{FG})
Let $G$ be a topological group and let $T\subset G$.
\begin{enumerate}
\item If the underlying topological space of $G$ is normal, then all  discrete closed subsets of   $G$ are approximable $\CB(G)$-interpolation sets.
\item \label{i}If $T$ is  right (resp. left) uniformly discrete ,
then  $T$ is an approximable $\luc(G)$-interpolation set (resp.
$\ruc(G)$-interpolation set).
\item If $G$ is assumed to be  metrizable, then  every $\luc(G)$-interpolation set (resp. $\ruc(G)$-interpolation set) is right (left) uniformly discrete.
\item \label{iii} If $G$ is an $E$-group and $T$ is  an $E$-set in $G$ which is right (or left) uniformly discrete with respect to $U^2$  for some neighbourhood $U$ of the identity such that $UT$ is translation-compact, then $T$ is an approximable $\wap_0(G)$-interpolation set.
    \item If $G$  is a metrizable $E$-group,   $T\subset G$ is a $\wap(G)$-interpolation set if and only if  $UT$ is translation-compact for some compact neighbourhood $U$ of the identity such that $T$ is right (or left) uniformly discrete with respect to $U^2$.
\end{enumerate}
\end{lemma}

 The following Lemma will be needed later on in Section 5.

  \begin{lemma}
   \label{lem:liftTC}
   Let $G$ be a locally compact group, let $H$ be a closed subgroup of $G$ and let $T\subset H$.
   \begin{enumerate}
   \item If $UT$ is  a
   right $t$-set in $H$ for some compact neighbourhood $U$ of the identity $e$ in $H$, then there is a compact neighbourhood $V$ of $e$ in $G$ such that $VT$ is a right $t$-set in $G$.
   \item If in addition $T$ is central, then the left analogue of Statment (i) holds also.
   \end{enumerate}
 \end{lemma}

 \begin{proof} Let $U$ be a compact neighbourhood of $e$ in $H$, and suppose that $UT$ is a right t-set in $H$.
 By definition there is a compact subset $K\subseteq H$ such that $gUT\cap UT$ is relatively compact whenever $g\notin K$.
 Let $V$ be a compact symmetric neighbourhood of the identity in $G$ such that $V\cap H= U$ and let $K^\prime=VKV$.

 Let $g\in G$ but $g\notin K^\prime$, and consider  a net $(g_\alpha)$ in $gVT \cap VT$ with no convergent subnet.
 Then,  for each $\alpha$, there are $v_\alpha, w_\alpha \in V$ and $t_\alpha, s_\alpha \in T$ such that
 $g_\alpha=v_\alpha t_\alpha=gw_\alpha s_\alpha,$ and so
 \[t_\alpha s_\alpha^{-1}= v_\alpha^{-1}gw_\alpha\quad\text{for every}\quad\alpha.\]
 Note that neither of the nets $(s_\alpha)$ and $(t_\alpha)$ has a convergent subnet since $V$ is compact.
  We can assume that $(v_\alpha)$ and $(w_\alpha)$ converge to $v,w\in V$, respectively.  Therefore $(t_\alpha s_\alpha^{-1})$
  is a net in $H$ which  converges to  $h=v^{-1}gw$. Since $H$ is closed, $h\in H$.
  Therefore, the net $(t_\alpha s_\alpha^{-1})$ is eventually in $(h V)\cap H=h (V\cap H)=h U$. This means that
  the net  $(t_\alpha)$ may be seen in $hUT \cap UT $, and therefore $hUT \cap U T$ is not relatively compact. But that would imply that $h\in K$, and so $g\in K^\prime $, whence a contradiction.
  Thus, $VT$ is a right $t$-set in $G.$

 To prove the analogous statement for left $t$-sets, we suppose in addition that $T$ is central and put again $K^\prime=VKV.$
 Using the fact that $T$ is central, the same argument leads to   \[s_\alpha^{-1}t_\alpha= w_\alpha gv_\alpha^{-1}\quad\text{for every}\quad\alpha.\]
 By taking subnets,  we see that $(t_\alpha)$ is eventually in $TUh\cap TU=UTh\cap UT$ with $h\in VgV$. Thus, $h$ must in $K,$ and so $g$ is in $K^\prime.$
 \end{proof}

To state a well-known necessary condition for a subset $T\subseteq G$ to be a $\B(G)$-interpolation set  we need the concept of large squares that we recall from \cite[Definition 3.3]{chou82}:

A finite subset $F$ of
 $G$ is an {\it $n$-square} if $F=AB$ where $|A|=|B|=n$ and $|F|=n^2$.
A subset of a group  is then said to {\it contain large squares}
 if it contains an $n$-square for every $n\in\mathbb N$.

Large squares are incompatible with Sidon sets, as proved  in \cite[Proposition 3.4]{chou82}. We restate here this theorem, stressing on $\B(G)$.

\begin{theorem}\label{B(G)} (Chou, \cite{chou82})
  Let $G$ be a topological group.
  A $\B(G)$-interpolation set  cannot contain large squares.
  \end{theorem}

  \begin{proof} Let $G_d$ be the group $G$ with the discrete topology.
  If $T$ is a $\B(G)$-interpolation set, then $T$ is also a $\B(G_d)$-interpolation set. By  Remark  \ref{b(g)=sidon}, $T$ is then  a $B(G_d)$-interpolation set and, by \cite[Proposition 3.4]{chou82}, $T$ cannot contain large squares.
  \end{proof}

\section{ The quotients of $\wap_0(G)$  by $ C_0(G)$ and\\  $\wap(G)$ by $\ap(G)\oplus C_0(G)$}\label{wapapc0}
In \cite{chou75}, Chou considered $E$-groups and proved that the
quotient space $\wap_0(G)/C_0(G)$ contains a linear isometric copy of
$\ell_\infty$.
In this section, we strengthen this result and prove  that if $G$ is an
$E$-group, then there is a linear isometric copy of
$\ell_\infty(\kappa)$ in the quotient $\wap_0(G)/C_0(G)$
where $\kappa$ is the compact covering number of an
$E$-set contained in $G.$ In particular, $\kappa=\kappa(G)$ when $G$ is an  $SIN$-group.

Our method applies further to show  that the quotient $\wap(G)/\ap(G)\oplus C_0(G)$ is non-separable.

\begin{theorem}\label{theor:w0c0}
%\label{theor:c0plusap}
Let $G$ be a non-compact locally compact $E$-group having an
$E$-set $X$ with a   compact covering number $\kappa$. Then the
quotient space $\wap_0(G)/C_0(G)$ contains a linear isometric
copy of  $\ell_\infty(\kappa)$.
\end{theorem}

\begin{proof}
Let $V$ be a fixed compact symmetric neighborhood of $e$. Then we consider a set $T\subset X$ as that constructed in Section 2 of \cite{F3}. This set has the following properties:
\begin{enumerate}
  \item $\kappa(T)=|T|=\kappa(X)$.
  \item $T$ is right $V^2$-uniformly discrete.
  \item $VT$ is a $t$-set.
\end{enumerate}
For completeness, we recall from \cite{F3} the construction of the set $T$ since this shall be needed in the proof.
We may assume that $e\in X$ and start with $x_0=e.$.
Suppose that the elements $x_\beta$ have been selected for all
$\beta<\alpha$ with $\alpha<\kappa$. Set $$X_\alpha= \bigcup\limits_{\beta_1,\beta_2,\beta_3<\alpha}V^2x_{\beta_1}^{\epsilon_1}x_{\beta_2}^{\epsilon_2}V^2x_{\beta_3}^{\epsilon_3}V,$$
where each $\epsilon_i=\pm1.$ Since $\kappa(X_\alpha)<\kappa,$ we pick $x_\alpha$ in $X\setminus X_\alpha$ for our set $T$. In this way, we form a set $T=\{x_\alpha:\alpha<\kappa\}$.

We obtain from Lemma
\ref{lem:lucint} that $T$ is   an approximable
$\wap_0(G)$-interpolation set. Since every  infinite subset of $T$ is uniformly discrete and $C_0(G)$-interpolation sets  must be relatively compact, and so finite (see the proof of Proposition 3.3 of \cite{FG})
%{\tt need to add $C_0(G)$-interpolation to this lemma},
%{\tt this is deleted:  have compact closure and infinite subsets of $T$ are uniformly discrete,}
any decomposition
$T=\bigcup_{\eta <\kappa} T_\eta$ as a disjoint union of $\kappa$-many infinite subsets leaves  us in position to  apply Theorem \ref{main:constr} and finish the proof.
\end{proof}

We deduce first that  the quotient $\wap(G)/\left(\ap(G)\oplus C_0(G)\right)$
contains an isomorphic copy of $\ell_\infty(\kappa)$.

\begin{corollary}\label{cor:wapapc0}
Let $G$ be a non-compact locally compact $E$-group having an
$E$-set $X$ with a  compact covering number $\kappa$. Then the
quotient space $\wap(G)/\left(\ap(G)\oplus C_0(G)\right)$ contains an \emph{isomorphic}
copy of  $\ell_\infty(\kappa)$.
\end{corollary}

\begin{proof}
We only have to recall that $\wap(G)=\ap(G) \oplus \wap_0(G)$, \cite{deL1}. Therefore
$\wap(G)/(\ap(G)\oplus C_0(G))$ is isomorphic to $\wap_0(G)/C_0(G)$, and apply then Theorem \ref{theor:w0c0}.
\end{proof}

If we want to use our Theorem \ref{main:constr} to obtain a linear \emph{isometric} copy of $\ell_\infty(\kappa)$ in the quotient $\wap(G)/\ap(G)$ we need  first an approximable $\wap(G)$-interpolation set which is not an $\ap(G)$-interpolation set. If $G$, for instance, is discrete this means we need a translation-finite set that is not an $I_0$-set. Such sets can be easily found in $\Z$, the additive group of integers: $T=\{3^n+n\colon n\in \N\}\cup\{3^n\colon n\in \N\}$  is such an example, see \cite[Example 1.5.2]{GH} for a (simple) proof. For arbitrary discrete groups, an example as simple as that  has escaped to us. A considerably  more complicated construction can be used to obtain an approximable  $\wap(G)$-interpolation set that is not a $\B(G)$-interpolation set, \emph{a fortiori} not an  $\ap(G)$-interpolation set when $G$ is an IN-group, a nilpotent group or a group with large enough centre,  see Section 5  (note that the set $T$ above \emph{is} a Sidon set, i.e., a $B(G)$-interpolation set).

 We present however, on the lines of Theorem \ref{main:constr}, an \emph{ad-hoc} construction  of a linear isomorphism  of $\ell_\infty(\kappa)$ into
  $\wap(G)/(\ap(G)\oplus C_0(G))$ with norm at most one whose inverse has norm at most $2$. A detailed look at the proof  reveals that it is actually based in finding an approximable
  $\wap(G)$-interpolation set that is not an \emph{approximable} $\ap(G)$-interpolation set. What makes this construction different from our general approach is that this set could even be an $\ap(G)$-interpolation set. Recall that in a non-compact locally compact group no $\ap(G)$-interpolation set is approximable.

%We are not able to obtain our linear isometric copy of $\ell_\infty(\kappa)$ in the quotient $\wap(G)/\ap(G)$ for any $E$-group. The main obstacle is that we need first an approximaple $\wap(G)$-interpolation  set
%which is not an $\ap(G)$-interpolation set. If $G$ is discrete for instance, this means that we need a translation-finite set which is not an $I_0$-set. This set will be of course either a Sidon set or not.
%The second situation, as already mentionned, is  a result due to Chou \cite{chou82}, and shall be strengthened in Section 5.
%In fact, when $G$ is $IN$-group or a nilpotent group, we obtain a finer result showing there is a linear isometric copy of $\ell_\infty(\kappa)$ in $\wap(G)/\B(G).$
%The first  situation means that we must find a set which is a Sidon set but not an $I_0$-set.
%  This is possible when $G$ is the additive group of the integers $\Z $(see \cite{??}),
%the set $T$ might be chosen as $T=\{2^n:n\in\N\}\cup\{2^n+n:n\in\N\}.$ {\tt is this correct?}.
%But in general, we cannot see how to to construct such a set.
%
%

%Denote by $\|.\|_q$ the quotient norm in the quotient space $\wap(G)/(\ap(G))\oplus C_0(G))$.
%Naturally, $\|f+\ap(G\oplus C_0(G))\|_q\le \|f\|$ for every $f\in \wap(G).$
%
%With the functions $f_{\mathbf{c}}$ as defined above, we show that \[\|f_{\mathbf{c}}+\ap(G)\oplus C_0(G)\|_q\ge \frac{\|\mathbf{c}\|}2.\] It is clear then that the quotient is non-separable.

\begin{theorem} Let $G$ be a non-compact, locally compact E-group having an E-set $X$ with a compact covering number $\kappa$. Then the Banach space $\wap(G)$ contains a linear isometric copy $L$ of  $\ell_\infty(\kappa)$
such that \[\frac{\|f\|}2\le\|f+\ap(G)\oplus C_0(G)\|_q\le\|f\|\quad\text{for every}\quad f\in L.\]
%\[\inf\{\|2f+g\|:g\in AP(G)\oplus C_0(G)\}\ge\|f\|\quad\text{for every}\quad f\in L.\]
In particular,
the quotient space $\wap(G)/(\ap(G)\oplus C_0(G))$ is non-separable.
\end{theorem}

\begin{proof} Let $V$ be a fixed compact symmetric neighbourhood of $e$ in $G$,
$T$ be the approximable $\wap(G)$-interpolation set used in Theorem \ref{theor:w0c0}
and $\{T_\eta:\eta <\kappa\}$ be any partition of $T$ into $\kappa$-many infinite subsets.

We begin with a function $f\in \wap(G)$ with $f(G\setminus VT)=\{0\}$ and  $f(T)=\{1\}$.
%As in the proof of Theorem \ref{main:constr} (or see \cite[Lemma 4.6]{FG} for this particular case),
%let $f\in \wap(G)$ be the function given by Lemma \ref{lem:unic},
%and for each
%$\mathbf{c}=(c_\eta)_{\eta<\kappa}\in\ell_\infty(\kappa)$,
Consider  the
function $f_{\mathbf{c}}\in \wap(G)$,  supported in $VT$, defined  in Theorem \ref{main:constr}.
% Moreover, $f$ may be chosen so that $f(x)=1$ for every $x\in T.$
Clearly, the map
\[
\mathbf{c} \mapsto f_{\mathbf{c}}:\ell_\infty(\kappa)\to \wap(G)\]
 is a linear isometry, so
we only need to check that the quotient map
\[\Psi \colon \ell_\infty (\kappa) \to \wap(G)/(\ap(G)\oplus C_0(G)\]
% given by
%\[\Psi(\mathbf{c})=f_{\mathbf{c}}+\ap(G)\oplus C_0(G)\quad\text{for every}\quad \mathbf{c}\in\ell_\infty(\kappa)\] is.

satisfies $\|2f_{\mathbf{c}}+\ap(G)\oplus C_0(G)\|_q\ge \|\mathbf{c}\|$ for every $\mathbf{c}\in\ell_\infty(\kappa)$.
%Let $f_a\in WAP(G)$ with $a\in \ell_\infty(\kappa)$.
Without loss of generality, we may assume that $\|\mathbf{c}\|=1.$
We claim that $\|2f_{\mathbf{c}}+g+h\|\ge 1$ for every $g+h\in \ap(G)\oplus C_0(G).$
Suppose, otherwise, that $\|2f_{\mathbf{c}}+g+h\|<1$ for some $g+h\in \ap(G)\oplus C_0(G),$ pick $\epsilon>0$ such that $\|2f_{\mathbf{c}}+g+h\|<1-\epsilon.$
Then, in particular, $$|2f_{\mathbf{c}}(x)+g(x)+h(x)|<1-\epsilon\quad\text{for every}\quad x\in
T.$$
Fix $\eta<\kappa$ such that $1-\epsilon/2<|c_\eta|.$

Since $h\in C_0(G),$ we may fix as well $x\in T_\eta$ such that $|h(x)|<\epsilon/4.$
Let $(x_n)_{n\ge1}$ be any sequence in $T_\eta$ with $x^{\epsilon}\ne x_n$ for every $n\in \mathbb N$, $\epsilon=\pm1$.
Suppose  $xx_n^{-1}x_m=vx_\beta\in VT$ for some $\beta<\kappa$ and $v\in V.$
By the definition of $T$, this is possible for at most two $m's$:
$x_{\beta}=x^{\epsilon}$ and so $x_m=x_nx^{-1}vx^{\epsilon}$, or $x_{\beta}=x_n^{\epsilon}$ and so
$x_m=x_nx^{-1}vx_n^{\epsilon}$.
In other words, for every fixed $n\in \mathbb N$, there exists at most two $m$ for which $x x_n^{-1}x_m\in VT.$
Therefore, for every fixed $n$, $f_\eta(x x_n^{-1}x_m)=0$ for every $m$ except maybe for these two $m's$.
Moreover, since for each $n$, the set $\{x x_n^{-1}x_m:m\in\mathbb N\}$ is not relatively compact,
we may choose $m$ such that $|h(x x_n^{-1}x_m)|<\epsilon/4$.

Now since $g\in \ap(G)$,  by taking subsquences if necessary, we may fix $n_0\in \mathbb N$ such that  that \[
\|r_{x_n}g-r_{x_m}g\|<\epsilon/2
%,\quad |h(x)|<\epsilon/4\quad\text{and}\quad  |h(x_mx_n^{-1}x)|<\epsilon/4
\quad\text{for every} \quad n,\,m\ge n_0;\]
and so, in particular,
\begin{align*}|g(x)-g(xx_n^{-1}x_m)|=|g(xx_n^{-1}x_n)-g(xx_n^{-1}x_m)|&=|r_{x_n}g(xx_n^{-1})-r_{x_m}g(xx_n^{-1})|\\&
\le \|r_{x_n}g-r_{x_m}g\|<\epsilon/2.\end{align*}

Therefore, for every fixed $n\ne n_0$ and $m\ge n_0$ chosen suitably, we have
\begin{align*}
2-\epsilon<2|c_\eta|&= |2c_\eta f_\eta(x)-2c_\eta f_\eta(xx_n^{-1}x_m)|=\\&=
|2f_{\mathbf{c}}(x)-2f_{\mathbf{c}}(xx_n^{-1}x_m)+g(x)-g(xx_n^{-1}x_m)+h(x)-h(xx_n^{-1}x_m)\\&-g(x)+g(xx_n^{-1}x_m)-h(x)+h(xx_n^{-1}x_m)|\\&\le|2f_{\mathbf{c}}(x)+g(x)+h(x)|+|2f_{\mathbf{c}}(xx_n^{-1}x_m)+g(xx_n^{-1}x_m)+h(xx_n^{-1}x_m)|\\&+|g(x)-g(xx_n^{-1}x_m)|+|h(x)-h(xx_n^{-1}x_m)|\\&
\le \|2f_{\mathbf{c}}+g+h\|+\|2f_{\mathbf{c}}+g+h\|+\|r_{x_n}g-r_{x_m}g\|+|h(x)|+|h(xx_n^{-1}x_m)|\\&
 \le(1-\epsilon)+(1-\epsilon)+\epsilon/2+\epsilon/2=2-\epsilon.
  \end{align*}
  This is clearly absurd, so we must have $\|2f_{\mathbf{c}}+g+h\|\ge1$ as required.
\end{proof}

Both theorems in this section will be considerably improved in the next section when $G$ is an $IN$-group or a nilpotent group.

\section{ The quotient of $\wap(G)$ by $\B(G)$}

The situation is much more  delicate with $\wap(G)/\B(G)$.
Already in the cases dealt with by Rudin in \cite{ru} and by Ramirez in \cite{ra} proving that $\B(G)\subsetneq
\wap(G)$ the arguments were quite involved.
Elaborating on the work by Rudin  and Ramirez, Chou proved in \cite{chou82} that the quotient space
$\wap(G)/\B(G)$ contains a linear isometric copy of $\ell_\infty$ whenever $G$ is a non-compact, locally compact, $IN$-group or a
 nilpotent group. In all these papers the key
argument consists in constructing  a $t$-set that contains  large
squares. We follow here  that thread  and find copies of
$\ell_\infty(\kappa)$ for $\kappa$ as large as possible in $\wap(G)/
\B(G)$ by applying Theorem \ref{main:constr}.

More precisely, we shall strengthen Chou's theorems by showing that there is a copy of
$\ell_\infty(\kappa)$ in the quotient $\wap(G)/\B(G)$ when $G$ is either an IN-group or a nilpotent group and $\kappa=\kappa(G)$, and that, in general, for  every locally compact group a copy of
$\ell_\infty(\kappa(Z(G)))$ can be found in the quotient $\wap(G)/\B(G).$

%It should be remarked that $t$-sets have been   used to study the
%algebra in the Stone-\v Compactification $\beta S$ when S is a
%discrete semigroup, see \cite{F1} and \cite{HS}. They came into play
%also with the study of $G^{\luc}$, see \cite{FP} and \cite{FS1}. It
%was again using t-sets that the topological centre of
%$L_\infty(G)^*$ was determined to be $L^1(G)$ and the topological
%centre of $\luc(G)^*$ was found to be the measure algebra $M(G)$,
 %see \cite{L}, \cite{LL}, \cite{FS2} and \cite{N}.{\tt Is this
%paragraph necessary?}

%Recall that a  locally compact group $G$ is an {\it $IN$-group} if it has an invariant neighbourhood of the identity.
\medskip

The following technical lemma establishes that a group cannot be
covered by $\beta$-cosets of finitely many different subgroups of
index larger than $\beta$. This is similar to a theorem, known at
least from the times of \cite{neum54}, in which only finitely many cosets are allowed.

 \begin{lemma} \label{lem:cover} Let $G$ be any group with $|G|=\kappa$. Suppose that
there is a finite collection $\{H_1,\ldots,H_n\}$ of subgroups of $G$
such that $G$ can be  covered by $\beta<\kappa$ right-cosets of
them, i.e., such that \begin{equation}\label{eq:cover}
G=\bigcup_{j=1}^n \bigcup_{i\in I_j} H_jx_{i,j} , \quad \mbox{ with
}|I_1|+\cdots+|I_n|=\beta<\kappa.\end{equation} Then some of the
subgroups $H_j$ has index at most $\beta$.
\end{lemma}

\begin{proof}
   This is proved by induction on $n$. The theorem is obvious if
$n=1$. Assume the theorem has been proved for unions of cosets of
$n-1$ different subgroups and suppose
\[
\label{11}
\tag{$\ast$} G=\bigcup_{j=1}^n \bigcup_{i\in I_j} H_jx_{i,j} , \quad \mbox{
with }|I_1|+\cdots+|I_n|=\beta<\kappa.\] If $|G:H_n|>\beta$, there
is $x\in G$ such that  $x\notin \bigcup_{i\in I_n} H_n x_{i,n}$.
%$x\notin \bigcup_{j=1}^n \bigcup_{i\in I_j} H_j x_{i,j}$,
Since \[H_n x\cap  H_nx_{i,n}=\emptyset\quad\text{ for
every}\quad i\in I_n,\] we obtain
\[ H_n x\subseteq\bigcup_{j=1}^{n-1}
\bigcup_{i\in I_j}  H_j x_{i,j},\]
and so
\[\label{2}
\tag{$\ast\ast$}H_n = \bigcup_{j=1}^{n-1}
\bigcup_{i\in I_j}  H_jx_{i,j}x^{-1} \cap H_n= \bigcup_{j=1}^{n-1}
\bigcup_{i\in I_j}  (H_j\cap H_n)y_{i,j},\] where the $y_{i,j}$'s
have been suitably chosen in $H_n$ (if $h_jx_{i,j}x^{-1} = y_{i,j}$ is in $ H_jx_{i,j}x^{-1} \cap H_n$, then $x_{i,j}x^{-1}=h_j^{-1} y_{i,j}$).
 Applying our inductive hypothesis, we
deduce that there is $j_0$ with $|H_n\colon (H_n\cap H_{j_0})|\leq \beta$.
(One may also proceed directly and replace $H_n$ from \eqref2 in \eqref{11}, then apply the inductive hypothesis).

There is therefore a family $\{z_s\colon s \in S\}\subset H_n$ with
$|S|\leq \beta$ such that $H_n=\bigcup_{s\in S} (H_n\cap H_{j_0})z_s$
and we may replace \eqref{eq:cover} by
\[ G=\left(\bigcup_{j=1}^{n-1} \bigcup_{i\in I_j}  H_jx_{i,j}\right)
\bigcup\left(\bigcup_{i\in I_n}\bigcup_{s\in S}
H_{j_0}z_sx_{i,n}\right).\] Since this is a cover of $G$ by cosets of at
most $n-1$ different  subgroups of $G$ we deduce from our inductive
hypothesis that some of the subgroups $H_j$, $1\leq j\leq n-1$, has
index at most $\beta$.
\end{proof}

\begin{lemma} \label{kappa} Let $G$ be a locally compact group
containing a   normal subgroup $H \subset G$. If $\left|G\colon H \right|=\kappa\geq \omega$, then
%If $H_1$ is a subgroup of $G$ such that $HH_1=H_1H$, then
 $G$ contains a family $\{T_\eta\colon
\eta<\kappa\}$ of  subsets such that, putting $T=\bigcup_{\eta<\kappa} T_\eta$,
\begin{enumerate}
\item $Ht\cap Ht^\prime=\emptyset$ for every $t\neq t^\prime \in T$.
%\item  $Ht\cap Ht^\prime=\emptyset$ for every $t\neq t^\prime \in T$.
%_\eta H\cap T_{\eta^\prime}H=\emptyset$ for $\eta,$ $\eta^\prime %<\kappa,${\displaystyle \eta\neq\eta',$
    \item  $T_\eta$ contains large squares for every $\eta<\kappa$.
\item  If  $U\subset H$ is  compact, then $UT g\cap UT$ and $gUT\cap UT$  are  relatively compact relatively compact for every $g\notin t^{-1}U^2t$ with $t\in T$.
\end{enumerate}
\end{lemma}

\begin{proof} For each  $\eta<\kappa$, let in  this proof
$\mathfrak{C}_\eta(H)$ denote the set \[\mathfrak{C}_\eta(H)=\{ g\in G
\colon \left|\mathfrak{Cl}(gH)\right|\leq \eta\},\] where
$\mathfrak{Cl}(gH)$ denotes the conjugacy class of $gH$ in $G/H$. If
$A\subseteq G$,  $\mathfrak{Cl}(A)$  will stand for the set
$\{g^{-1}ag\colon g\in G, a\in A\}$.
For $n<\omega$, we define   $\langle A\rangle_n$ to be  the set of all products of at most $n$ elements in $A\cup A^{-1}$.

%\emph{Case I: $ \kappa>\aleph_0$}

We  define for each $\eta <\kappa$ and $n<\omega$ two collections of
finite sets
\[ C_{\eta,n}=\left\{x_{\eta,n,k}\colon 1\leq k\leq n\right\} \quad \mbox{ and }
\quad  D_{\eta,n}=\left\{y_{\eta,n,k}\colon 1\leq k\leq n\right\}.\]

These sets are defined recursively. First we order the set
 $\kappa\times
\omega$ in the canonical way (\cite[3.12]{jech}): For $\eta,\eta^\prime <\kappa$ and $n,n^\prime<\omega$, we define  $(\eta,n)<(\eta^\prime,n^\prime)$ if either
$\max(\eta,n)<\max(\eta^\prime,n^\prime)$, or
$\max(\eta,n)=\max(\eta^\prime,n^\prime)$ and $\eta<\eta^\prime$,
or $\max(\eta,n)=\max(\eta^\prime,n^\prime)$, $\eta=\eta^\prime $ and $n<n^\prime$.
 This way, the cardinal of the set $\left\{(\eta,n)\colon (\eta,n)<(\eta_0,n_0)\right\}$ is  less than $\kappa$ for every $(\eta_0,n_0)\in \kappa\times \omega$.

We now define $x_{1,1,1}=y_{1,1,1}= e$  and set $C_{1,1}=\{x_{1,1,1}\}$, $D_{1,1}=\{y_{1,1,1}\}$.

Assume that the sets
$C_{\gamma,j}$
and $D_{\gamma,j}$ and the elements $x_{\eta,n,l}$ have been defined for $(\gamma,j)<(\eta,n)$ and $1\leq l<k$. We describe how to define $x_{\eta,n,k}$. Define first the following subsets of $G/H$ by
\[ R_{\eta,n,k}=\bigcup_{(\gamma,j)<(\eta,n)}
\left(C_{\gamma,j}H \cup D_{\gamma,j}H\right) \bigcup \left(\bigcup_{l=1}^{k-1} \,x_{\eta,n,l}H\right)
\] and
\[ S_{\eta,n,k}=\left(\bigcup_{(\gamma,j)\leq(\eta,n)} C_{\gamma,j}H \right)\bigcup \left(\bigcup_{(\gamma,j)<(\eta,n)} D_{\beta,i}H \right)\bigcup \left(\bigcup_{l=1}^{k-1} \,y_{\eta,n,l}H\right).\]

Consider a cardinal number $f(\eta,n,k)$ such that
\[\kappa > f(\eta,n,k)>n^2|\langle R_{\eta,n,k}\rangle_8 |+|\langle S_{\eta,n,k}\rangle_8|.\]
 Note that, in particular, $f(\eta,n,k)$ is finite when $\kappa=  \omega$.

We now  choose
$x_{\eta,n,k}$ such that
\begin{equation}\label{eq:x1} x_{\eta,n,k}H
 \notin \Big\langle R_{\eta,n,k}, \mathfrak{Cl}\left(R_{\eta,n,k}\cap
\mathfrak{C}_{f(\eta,n,k)}\right)\Big\rangle_8.\end{equation}
%If $\kappa=\omega$,  $\langle R_{\eta,n,k}, \mathfrak{Cl}\left(R_{\eta,n,k}\cap
%\mathfrak{C}_\eta(H)\right)\Big\rangle$ is finite, when $\kappa>\aleph_0$
%we have that
Finding this
element is possible  because
$|\langle R_{\eta,n,k}\rangle_8 H/H|\leq f(\eta,n,k)<\kappa=|G\colon H|$ and conjugacy classes of elements of $\mathfrak{C}_{f(\eta,n,k)}$ do not have, by definition, more than $f(\eta,n,k)$-elements.
%\[\Bigl|\mathfrak{Cl}(R_{\eta,n,k}\cap
%\mathfrak{C}_{f(\eta,n,k)}(H))H/H\Bigr|= \left|\bigcup_{x\in R_{\eta,n,k}\cap
%\mathfrak{C}_{f(\eta,n,k)} \mathfrak{Cl}(x)H/H\right|\leq \eta<|G:H|.\]

Once the elements in $C_{\eta,n}$ are  defined in this way, we
define the elements in $D_{\eta,n}$. If $y_{\eta,n,1},\ldots, y_{\eta,n,k-1}$ have already been defined, we use the following Claim to define $y_{\eta,n,k}$.

\textbf{Claim 1:} \label{cl1} \emph{There is   $y_{\eta,n,k}\in G$ such that
\begin{equation}\label{eq:y1} y_{\eta,n,k}H \notin \langle S_{\eta,n,k}\rangle_8 \end{equation}
and
\begin{equation}\label{eq:y2}
C_{\eta,n} C_{\eta,n}^{-1}\bigcap
y_{\eta,n,k}R_{\eta,n,k}y_{\eta,n,k}^{-1} H=\{e\}.\end{equation}
}

We first enumerate
$(C_{\eta,n}C_{\eta,n}^{-1}\setminus\{e\}) \bigcap
  \mathfrak{Cl}\left(R_{\eta,n,k}\right)H$ as $ \{a_1,\ldots,a_l\}$.
Let then \[R_j=\{r\in R_{\eta,n,k} \colon rH\in \Cl(a_jH)\} \] and choose   for each $j$, $1\leq j \leq l$, and
each $r\in R_j$
     an element $y_{j,r}\in G$ and an element $h_{j,r}\in H$
     with
     \[ r= y_{j,r}^{-1} a_j  h_{j,r}
     y_{j,r}.\]

Suppose now that no $y\in G$ can be found so that conditions
   \eqref{eq:y1} and \eqref{eq:y2} are satisfied.
In that case some $R_j$ must be  non-empty and, indeed,
% Observe as well that   $y^{-1}a_j y=y_{j,r}^{-1}a_j y_{j,r}$
%implies that $y\in  C_G(a_j)y_{j,r}$.
\[G=\langle S_{\eta,n,k}\rangle_8 \bigcup \left(\bigcup_{j=1}^l \bigcup_{r\in R_j}
 \bigcup_{h\in H}
L_{j,r,h}\right),\]
 where
 \[ L_{j,r,h}=\left\{g\in G \colon r=g^{-1}a_j h g\right\}.\]

Observe now that
that $L_{j,r,h}H\subseteq C_{G/H}(a_jH)y_{j,r}H$, where \[C_{G/H}(a_jH)=\{ gH\in
G/H\colon ga_jH=a_jgH\}\] is the centralizer of $a_jH$. Therefore,
 \begin{equation}
   \label{eq:normal}
G/H= \langle S_{\eta,n,k}\rangle_8\, \bigcup \left(\bigcup_{j=1}^l \bigcup_{r\in R_j}
C_{G/H}( a_jH)y_{j,r}H\right). \end{equation}
If the  elements of the  set  $\langle S_{\eta,n,k}\rangle_8$ are  viewed as cosets of the trivial subgroup $\{e_{G/H}\}$,  we find $G/H$ as a union  of less than $n^2|R_{\eta,n,k}|+|\langle S_{\eta,n,k}\rangle_8|$ cosets. Since the latter number is less than $f(\eta,n,k)$  some of them must correspond to a subgroup of index at most $f(\eta,n,k)$,  by Lemma \ref{lem:cover}.
Thus, there is $j$, $1\leq j\leq n$ such that $|G/H\colon C_{G/H}(a_jH)|=|\mathfrak{Cl}(a_jH)|<f(\eta,n,k)$.
 We conclude that  $a_j \in \mathfrak{C}_{f(\eta,n,k)}$.

 Since
$a_jH \cap
  \mathfrak{Cl}\left(R_{\eta,n,k}\right)\neq \emptyset $, we find that
   $a_j\in  \mathfrak{Cl}(R_{n,\eta,k}\cap
\mathfrak{C}_{f(\eta,n,k)})H$. If $ a_j=x_{\eta,n,k_1}^{-1}x_{\eta,n,k_2}$
and $k_2>k_1$, this goes against condition \eqref{eq:x1} in the
choice of $x_{\eta,n,k_2}$ and finishes  the proof of the claim.

We have therefore constructed two families
\[C_{\eta,n}=\left\{x_{\eta,n,k}\colon 1\leq k\leq n\right\}\;\text{ and}\;
D_{\eta,n}=\left\{y_{\eta,n,k}\colon 1\leq k\leq n\right\},\quad
(\eta,n)\in \kappa\times \omega,\] with properties \eqref{eq:x1}, \eqref{eq:y1} and  \eqref{eq:y2}.
 We now check that the sets
\[T_\eta=\bigcup_{n<\omega}(D_{\eta,n}C_{\eta,n}),\quad\eta<\kappa\]  satisfy the
desired properties.
%It is useful to observe that $T_\eta \cap
%\langle T_{\gamma}\colon \gamma<\eta\rangle =\emptyset$.

First of all we see that $|D_{\eta,n}C_{\eta,n}|=n^2$. If this were
not
 the case, there would be  $1\leq k_1,k_2,k_3,k_4\leq n$
 with $k_1\neq k_3$ and $k_2<k_4$ such that
$y_{\eta,n,k_1}x_{\eta,n,k_2}=y_{\eta,n,k_3}x_{\eta,n,k_4}$. But
then $y_{\eta,n,k_4}\in \langle C_{\eta,n},y_{\eta,n,k_2}\rangle_2$ which goes
against our choice of the elements in $D_{\eta,n}$.
Therefore, $T_\eta$ contains $n$-squares for every $n$.

%
% We deduce from the preceding paragraph that $T_\eta$ contains
%$n$-squares and therefore, by  LTheorem \ref{B(G)}, that $T_\eta$ is
%not a $\B(G)$-interpolation set.

%We now see that $T$ is a $t$-set.
In order to prove the last statement, we take  $U$  a compact subset of $H$ and $ g\notin U^2$. Choose
$t_1=y_{\eta_1,n_1,k_1}x_{\eta_1,n_1,k_1^\prime}$ and $t_2=
y_{\eta_2,n_2,k_2}x_{\eta_2,n_2,k_2^\prime}$, $u_1, u_2\in U$, with
\[ gu_1t_1=u_2t_2\in gUT\cap UT.\]

We order the  3-tuples $(\eta_i,n_i,k_i)$ lexicographically with respect to the last entry, that is
$(\eta_i,n_i,k_i)>(\eta_i^\prime,n_i^\prime,k_i^\prime)$ if either
$(\eta_i,n_i)>(\eta_i^\prime,n_i^\prime)$ or $(\eta_i,n_i)=(\eta_i^\prime,n_i^\prime)$ and $k_i>k_i^\prime$.

Assume that $(\eta_2,n_2)\geq (\eta_1,n_1)$.

Let now $gu_3t_3=u_4t_4$ with $t_3=y_{\eta_3,n_3,k_3}x_{\eta_3,n_3,k_3^\prime}$ and $t_4=y_{\eta_4,n_4,k_4}x_{\eta_4,n_4,k_4^\prime}$,
$u_3,u_4\in U$
be any other element of $ gUT\cap UT$.

%\begin{equation}\label{ln} gu_3 y_{\eta_3,n_3,k_3}x_{\eta_3,n_3,k_3^\prime}=
%u_4y_{\eta_4,n_4,k_4}x_{\eta_4,n_4,k_4^\prime}, \quad u_3,u_4\in U\end{equation}

Then
\begin{align*}
  g&=u_4y_{\eta_4,n_4,k_4}x_{\eta_4,n_4,k_4^\prime}
x_{\eta_3,n_3,k_3^\prime}^{-1}
y_{\eta_3,n_3,k_3}^{-1}u_3^{-1}\\&=u_2y_{\eta_2,n_2,k_2}x_{\eta_2,n_2,k_2^\prime}
x_{\eta_1,n_1,k_1^\prime}^{-1}y_{\eta_1,n_1,k_1}^{-1}u_1^{-1}.\end{align*}

Let $(\eta_{i_1},n_{i_1},k_{i_1})=\max\{(\eta_i,n_i,k_i)\colon 1\leq
i \leq 4\}$. There must then be $i_2$, $1\leq i_2\leq 4$, $i_2\neq
i_1$, with $\eta_{i_1}=\eta_{i_2}$, $n_{i_1}=n_{i_2}$ and $
y_{\eta_{i_1},n_{i_1},k_{i_1}} =y_{\eta_{i_2},n_{i_2},k_{i_2}}$, for
otherwise $y_{\eta_{i_1},n_{i_1},k_{i_1}}\in
S_{\eta_{i_1},n_{i_1},k_{i_1}}$.

%\textbf{Case 1}. \emph{Suppose  that $(\eta_4,n_4)>
%(\eta_3,n_3)>(\eta_2,n_2)$} (again the order is lexicographic).
%
%In that case $y_{\eta_4,n_4,k_4^\prime} \in S_{\eta_4,n_4}$, against
%our condition \eqref{eq:y1} in the election of
%$y_{\eta_4,n_4,k_4^\prime}$.
 \textbf{Claim 2}: \emph{It is not possible that $(\eta_4,n_4)=
(\eta_3,n_3)>(\eta_2,n_2)$}. Should this be the case, then
$y_{\eta_4,n_4,k_4}=y_{\eta_3,n_3,k_3}$ and
\[ x_{\eta_3,n_3,k_4^\prime}x_{\eta_3,n_3,k_3^\prime}^{-1}\in
y_{\eta_3,n_3,k_3}^{-1}\left(y_{\eta_2,n_2,k_2}
x_{\eta_2,n_2,k_2^\prime}x_{\eta_1,n_1,k_1^\prime}^{-1}
y_{\eta_1,n_1,k_1}^{-1} \right)y_{\eta_3,n_3,k_3}
H.\] It
follows from  our
condition \eqref{eq:y2} in the choice  of
$y_{\eta_3,n_3,k_3}$ (Claim 1, page \pageref{cl1}) that
$x_{\eta_3,n_3,k_4^\prime}=x_{\eta_3,n_3,k_3^\prime}^{-1}$
but this is only possible if   $g\in U^2$ (take into account hat
$gu_3t_3=u_4t_4$ and that
 $y_{\eta_4,n_4,k_4}=y_{\eta_3,n_3,k_3}$), and the claim is proved.

The same argument shows that it is not possible that $(\eta_1,n_1)=(\eta_2,n_2)>(\eta_i,n_i)$,
with $i=3,4$.

We deduce that either
$(\eta_4,n_4)=(\eta_2,n_2)$ or $(\eta_3,n_3)=(\eta_2,n_2)$. But,
since the element $ gu_3t_3=u_4t_4$ was chosen arbitrarily in $gUT\cap UT$, it
 follows that
\[ g UT\cap UT \subseteq
(g U D_{\eta_2,n_2}C_{\eta_2,n_2})\cup (U
D_{\eta_2,n_2}C_{\eta_2,n_2}),\] and this is a relatively compact
set.

We  now prove check  that $  UTg\cap UT $ is compact. Choose
$t_1=y_{\eta_1,n_1,k_1}x_{\eta_1,n_1,k_1^\prime}$ and $t_2=y_{\eta_2,n_2,k_2}x_{\eta_2,n_2,k_2^\prime}$, $u_1, u_2\in U$
with \[ u_1t_1g=u_2t_2\in UTg\cap UT.\]

Let   \[ u_3 y_{\eta_3,n_3,k_3}x_{\eta_3,n_3,k_3^\prime}g=
u_4y_{\eta_4,n_4,k_4}x_{\eta_4,n_4,k_4^\prime}, \quad u_3,u_4\in U\]
be any other element of $ UTg\cap UT$ with $(\eta_3,n_3)\geq (\eta_4,n_4)$. We have that
\begin{align*}
  g&=t_1^{-1} u_1^{-1}u_2
t_2\\&=x_{\eta_3,n_3,k_3^\prime}^{-1}y_{\eta_3,n_3,k_3}^{-1}
u_3^{-1}u_4 y_{\eta_4,n_4,k_4}x_{\eta_4,n_4,k_4^\prime}\in
x_{\eta_3,n_3,k_3^\prime}^{-1}y_{\eta_3,n_3,k_3}^{-1}
y_{\eta_4,n_4,k_4}x_{\eta_4,n_4,k_4^\prime} H.
\end{align*}

As in the preceding case we assume that $(\eta_2,n_2)\geq (\eta_1,n_1)$ and it is enough to see that neither $(\eta_4,n_4)=(\eta_3,n_3)>(\eta_2,n_2)$, nor
$(\eta_2,n_2)= (\eta_1,n_1)>(\eta_3,n_3)$.

If $(\eta_4,n_4)=(\eta_3,n_3)>(\eta_2,n_2)$, then necessarily $(\eta_4,n_4,k_4)=(\eta_3,n_3,k_3)$ and   $t_1^{-1}
u_1^{-1}u_2 t_2\in
x_{\eta_3,n_3,k_3^\prime}^{-1}x_{\eta_3,n_3,k_4^\prime} H$.
If $k_3^\prime >k_4^\prime$, then $x_{\eta_3,n_3,k_3^\prime}\in R_{\eta_3,n_3,k_3^\prime}$, against the election of $x_{\eta_3,n_3,k_3^\prime}$. But
  $k_3^\prime= k_4^\prime$, implies that $t_3=t_4$ (recall that $(\eta_4,n_4,k_4)=(\eta_3,n_3,k_3)$) and  $g\in t_3^{-1}U^2t_3 $.
We rule out analogously  the possibility $(\eta_1,n_1)=(\eta_2,n_2)$
and  argue as above to prove that
\[  UTg\cap UT \subseteq
( U D_{\eta_2,n_2}C_{\eta_2,n_2})g\cup (U
D_{\eta_2,n_2}C_{\eta_2,n_2}),\] and conclude that $UTg\cap UT$ is
relatively compact.
\end{proof}

\begin{corollary}\label{cor:kappa}
Let $G$ be a locally compact group, $H$ a subgroup of $G$ and $N_G(H)=\{ g\in G\colon gH=Hg\}$ be the normalizer of $H$ in $G$.
 If $\left|N_G(H)\colon H \right|=\kappa\geq \omega$, then
%If $H_1$ is a subgroup of $G$ such that $HH_1=H_1H$, then
 $N_G(H)$ contains a family $\{T_\eta\colon
\eta<\kappa\}$ of  subsets such that
\begin{enumerate}
%\item ${\displaystyle \kappa =\left|H_1:H\cap H_1\right|}$.
%\item ${\displaystyle \kappa =\left|G\colon H \right|}$.
%\item  $Ht\cap Ht^\prime=\emptyset$ for every $t\neq t^\prime \in T$.
\item if $T=\bigcup_{\eta<\kappa} T_\eta$, then $Ht\cap Ht^\prime=\emptyset$ for every $t\neq t^\prime \in T$;
%_\eta H\cap T_{\eta^\prime}H=\emptyset$ for $\eta,$ $\eta^\prime %<\kappa,${\displaystyle \eta\neq\eta',$
    \item  $T_\eta$ contains large squares for every $\eta<\kappa$;
\item if $T=\bigcup_{\eta<\kappa} T_\eta$ and $U\subset H$ is  compact, then $UT g\cap UT$ and $gUT\cap UT$  are  relatively compact relatively compact for every $g\notin t^{-1}{U^2}t$ with $t\in T$.
\end{enumerate}
\end{corollary}

\begin{proof}
  Since $H$ is a   normal subgroup of $N_G(H)$,  we can apply Lemma  \ref{kappa} to $N_G(H)$.  Statements  (i) and (ii) remain the same if $N_G(H)$ is replaced by $G$. As for Statement (iii), one notices that  $gUT\cap UT\neq \emptyset$ and  $UTg\cap UT\neq \emptyset$ both  imply that $g\in N_G(H)$, and so this statement follows also from Lemma  \ref{kappa}.
\end{proof}

We obtain a first consequence for groups with large center.

 \begin{theorem}
  \label{center}
 Let $G$ be  a locally compact group,   $Z(G)$ be the algebraic center of $G$ and put $\kappa=\kappa(Z(G))$. If $\kappa\neq 1$, then  there is always a linear isometry
 $\Psi \colon \ell_\infty(\kappa)\to \wap(G)/\B(G)$.
\end{theorem}

\begin{proof}
The center $Z(G)$, as every locally compact Abelian group, always contains an open subgroup $G_0$ topologically isomorphic to $\R^n\times K$, with $K$ compact.
Therefore, $G$ must contain   two subgroups $H_1\subset H_2\subseteq Z(G)$ with $H_1$ open in $H_2$
and $|H_2\colon H_1|=\kappa$. Indeed, if $|Z(G)\colon G_0|\geq \omega$, then  we take $H_2=Z(G)$ and $H_1=G_0$. If $G_0$ has finite index and $Z(G)$ is not compact, then $\kappa =\omega$ and so we may take $H_2=\Z$
and $H_1=\{e\}$.

 Let $\{T_\eta:\eta<\kappa\}$  be the family of subsets of $H_2$ provided by Lemma \ref{kappa}.   By Theorem \ref{B(G)}, none  of them   is  a $\B(G)$-interpolation set. If $T=\bigcup_{\eta<\kappa}
T_\eta$ and $U$ is a compact neighbourhood of the identity in $H_2$ with $U\subset H_1$, then, since $H_2$ is commutative,  $UT$ is a t-set in $H_2$. By Lemma \ref{lem:liftTC}, we can find a compact neighbourhood $V$ of the identity  in $G$, such that $VT$ is a t-set in $G$.  If $V$ is chosen so that $V^4\cap H_2\subset H_1$ (remember $H_1$ is open in $H_2$), then  $T$ is $V^2$-uniformly discrete, and by Lemma \ref{lem:lucint},  $T$   is an approximable $\wap(G)$-interpolation set.

It suffices now to apply Theorem \ref{main:constr}.
\end{proof}

Theorem \ref{kappa} can be readily applied to discrete groups.
To further expand its applicability we follow the usual path applying well-known structure theorems. The following Lemma for instance is the analog of  Lemma 4.4 of \cite{chou82}.

\begin{lemma}\label{quotquot}
  Let $G$ be a locally compact group and let $N$ be a closed  subgroup of $G$.
   \begin{enumerate}
   \item If $N$ is normal, the quotient map $\pi\colon G\to G/N$ induces linear isometries \begin{align*}&\Pi \colon \wap(G/N)/\B(G/N)\to \wap(G)/\B(G)\quad\text{and}\\&
    \Pi_0 \colon \wap_0(G/N)/\B_0(G/N)\to \wap_0(G)/\B_0(G).\end{align*}
   \item If $N$ is open, there are  linear isometries \begin{align*}&\Psi\colon \wap(N)/\B(N)\to \wap(G)/\B(G)\quad\text{and}\\&
   \Psi_0\colon \wap_0(N)/\B_0(N)\to \wap_0(G)/\B_0(G).\end{align*}
   \end{enumerate}
  \end{lemma}

  \begin{proof}
We first prove (i).
  The map $\phi \mapsto \phi \circ \pi$ clearly defines a linear isometry  $\tilde{\pi}\colon \wap(G/N)\to \wap(G)$.
  By \cite[Theorem]{chou79}, we have \[\tilde{\pi}(\B(G/N))=\tilde{\pi}(\CB(G/N))\cap \B(G).\]
  By \cite{Bu}, we have \[\tilde{\pi}(\wap(G/N))=\tilde{\pi}(\CB(G/N))\cap \wap(G).\] Since $\B(G)\subseteq\wap(G)$, we see that
  \begin{align}\label{star}\tilde{\pi}(\B(G/N))=\tilde{\pi}(\CB(G/N))\cap \wap(G)\cap\B(G)=\tilde{\pi}(\wap(G/N))\cap \B(G)\end{align}
   so that  $\tilde{\pi}$ induces  a linear isomorphism
\[ \Pi \colon \wap(G/N)/\B(G/N)\to \wap(G)/\B(G),\]
given by \[\Pi(\phi+\B(G/N))=\tilde{\pi}(\phi)+\B(G).\]
We check that $\Pi$ is an isometry.
If $\phi \in \wap(G/N)$,
\begin{align*}
\|\Pi(\phi+\B(G/N))\|&=\|\tilde{\pi}(\phi)+\B(G)\|\\
& =\inf\{\|\tilde{\pi}(\phi)+\psi\| \colon \psi \in \B(G)\}\\
&\leq  \inf\{\|\tilde{\pi}(\phi+\psi)\| \colon \psi \in \B(G/N)\}
\\
&=  \inf\{\|\phi+\psi\| \colon \psi \in \B(G/N)\} =\|\phi+\B(G/N)\|.
\end{align*}
For the reverse inequality,  we follow the path of Lemma 2.3 of \cite{chou80}
and consider  the invariant mean $\mu_N$ on $\wap(N)$.
For $\phi\in \wap(G)$, we define the function $\phi^N\colon G\to\C$ by
\[\phi^N(g)=\mu_N(\phi_g), \quad \mbox{where } \phi_g(h)=\phi(gh).\]
By invariance of $\mu_N$, the function $\phi^N$ is constant on the cosets of $N$ and therefore induces a continuous function on $G/N$. Clearly, $\|\phi^N\|_{G/N}\leq \|\phi\|_G$.

Now Lemma 2.3 of \cite{chou80} proves in fact that  $\phi^N\in \wap(G/N)$.  Moreover, by first considering positive-definite functions, it
 is also easily checked that  $\psi^N \in \B(G/N)$ for every   $\psi\in \B(G)$.

Note as well that
for $\phi \in \wap(G/N)$, we have $\tilde{\pi}(\phi)^N=\phi$.

Now if $\phi \in \wap(G/N)$ and $\psi \in \B(G)$,
\begin{align*}
%\|\Pi(\phi+\B(G/N))\|&=\|\tilde{\pi}(\phi)+B(G)\|\\
\| \tilde{\pi}(\phi)+\psi\|
&\geq \|(\tilde{\pi}(\phi)+\psi)^N\|\\
&=\|\tilde{\pi}(\phi)^N+\psi^N\|\\
&=\|\phi+\psi^N\|\\
&\geq \|\phi+\B(G/N)\|.
\end{align*}
And the remaining inequality  \[\|\Pi(\phi+\B(G/N))\|=\|\tilde{\pi}(\phi)+\B(G)\|\geq \|\phi+\B(G/N)\|\] follows. \medskip

We prove now the analogue statements for $\wap_0(G)$ and $\B_0(G).$
We check first that $\tilde{\pi}$ maps $\wap_0(G/N)$ into $\wap_0(G)$ and $\B_0(G/N)$ into $\B_0(G)$.
Consider the adjoint of $\tilde\pi,$ this is the map given by \[\tilde\pi^*:\wap(G)^*\to\wap(G/N)^*,\quad \tilde\pi^*(\nu)=\nu\circ\tilde\pi.\]
Note that if $\mu\in\wap(G)^*$ is invariant then $\tilde\pi^*(\mu)\in\wap(G/N)^*$ is invariant.
To see this, let $\bar s=\pi(s)\in G/N$ and $f\in\wap(G/N)$ and note that $\tilde\pi(f_{\bar s})=(\tilde\pi(f))_s,$ and so
\[\tilde\pi^*(\mu)(f_{\bar s})=\mu(\tilde\pi(f_{\bar s})=\mu((\tilde\pi(f))_s)=\mu(\tilde\pi(f))=\tilde\pi^*(\mu)(f).\]
Thus, $\tilde\pi^*(\mu)$ is the invaraint on $\wap(G/N).$

Let now $f\in \wap_0(G/N)$ and $\mu$ be the invariant mean on $\wap(G)$. Then $\tilde\pi(f)\in \wap(G)$ and
\[\mu(|\tilde\pi(f)|)=\mu(|f\circ\pi|)=\mu(|f|\circ \pi)=\tilde\pi^*(\mu)(|f|)=0.\] Thus, $\tilde\pi(f)\in\wap_0(G).$
To see that $\tilde\pi(f)\in \B_0(G)$
when $f\in \B_0(G/N)$, we argue in a similar way using the fact that $\tilde\pi(f)\in \B(G)$ by \cite[Theorem]{chou80}.
Accordingly, \[
\tilde\pi(\B_0(G/N))\subseteq \tilde\pi(\wap_0(G/N))\cap \B_0(G).\]

The reverse inclusion is checked as follows. If $f\in \tilde\pi(\wap_0(G/N))\cap \B_0(G),$ then by (\ref{star}) $f$ is clearly in $\tilde\pi(\B(G/N)).$
So let $g\in\B(G/N)$ with $f=\tilde\pi(g).$
We only need to make sure that $\tilde\pi^*(\mu)(|g|)=0.$
But this is also clear from the following identity.
\[\tilde\pi^*(\mu)(|g|=\mu(\tilde\pi(|g|))=\mu(|g|\circ\pi)=\mu(|g\circ\pi|)=\mu(|\tilde\pi(g)|)=\mu(|f|).\]
Thus, we obtain the analogue of (\ref{star})
  \begin{align}\label{starstar}\tilde{\pi}(\B_0(G/N))=\tilde{\pi}(\wap_0(G/N))\cap \B_0(G)\end{align}
   so that  $\tilde{\pi}$ induces  a linear isomorphism
  %By \cite[Theorem]{Chou}, $\tilde{\pi}(\B(G/N))=\tilde{\pi}(\wap(G/N))\cap \B(G)$ so that  $\tilde{\pi}$  induces  a linear isomorphism
\[ \Pi_0 \colon \wap_0(G/N)/\B_0(G/N)\to \wap_0(G)/\B_0(G),\]
given by \[\Pi_0(\phi+\B_0(G/N))=\tilde{\pi}(\phi)+\B_0(G).\]
To check that $\Pi_0$ is an isometry, we proceed precisely as for $\Pi$.

For the proof of (ii), we associate  to each $\phi\in \wap(N)$ the function \begin{equation}\label{ext}\phi_N(g)=\phi(g)\quad\text{ if}\quad g\in N\quad\text{ and} \quad
\phi_N(g)=0\quad\text{ if}\quad g\notin N.\end{equation}
Then  $\phi_N$ is in $\wap(G)$ by \cite[Lemma 5.4]{deL2},
\cite[Theorem 3.14]{Bu} or \cite[Lemma 2.4]{chou75}. If $\phi$ happens to be in $B(N)$  then $\phi_N\in B(G)$
by  \cite[page 280]{HR} or \cite[Lemma 4.1]{chou82} and this obviously extends to $\B(N)$ and $\B(G)$.

Define then
$\Psi\colon \wap(N)/\B(N)\to \wap(G)/\B(G)$ by \[\Psi(\phi+\B(N))=\phi_N+\B(G).\]
It is easy to check that $\Psi$ is a linear isometry.

To prove the second  statement of (ii), note that  the extension of $\phi_N$ defined in (\ref{ext}) is clearly in $\wap_0(G)$
 if $\phi\in \wap_0(N)$, and in $\B_0(G)$ if $\phi\in \B_0(N)$.
It is again straightforward to verify that
\[\Psi_0\colon \wap_0(N)/\B_0(N)\to \wap_0(G)/\B_0(G),\quad \Psi_0(\phi+\B_0(N))=\phi_N+\B_0(G)\]
is the required linear isometry.
 \end{proof}

 We reach finally our main results.

 \begin{theorem}\label{theorem:inquot} Let G be a
 a non-compact, locally compact, $IN$-group and put
$\kappa=\kappa(G)$.  Then there is a linear isometry $\Psi\colon \ell_\infty(\kappa)\to
\wap(G)/\B(G)$.
\end{theorem}

\begin{proof}
Let $G_0$ denote the connected component of $G$. By Theorem 2.13 of \cite{grossmosk71},
there is an open normal subgroup $N$ of $G$  that contains a compact normal subgroup $K$ with $N/K$ Abelian.

Suppose first that  $\kappa(N)=\kappa$. Then $\kappa=\kappa(N/K)$. By
Theorem \ref{center},  there is a linear isometric copy of $\ell_\infty(\kappa)$ in $\wap(N/K)/\B(N/K)$.  We apply then  (i) and (ii)  of Lemma \ref{quotquot} to obtain a linear isometric copy of $\ell_\infty(\kappa)$ in $\wap(G)/\B(G)$.

If $\kappa(N)<\kappa$, it follows that $\kappa=|G:N|$.
We apply  Lemma \ref{kappa} to the discrete group $G/N$. Let $\{T_\eta:\eta<\kappa\}$ be the collection of subsets obtained in that Lemma (in this case the subgroup $H$ of that Lemma is trivial)
 and let $T=\bigcup_{\eta<\kappa}T_\eta$.  By (iii) in that Lemma, the set $T$ is a  $t$-set (note that $U=\{e\}$ in this case, hence $Tg\cap T$ and $gT\cap T$ are finite if $g\neq e$) while each of the  sets $T_\eta$ contains large squares. Therefore,
  $T$ is a $\wap(G/N)(G)$-interpolation set by Lemma \ref{lem:lucint}, while   none  of the sets $T_\eta$  is  a $\B(G/N)(G)$-interpolation set by Theorem \ref{B(G)}.

By   Theorem \ref{main:constr} there is  an
    isometric embedding \[\ell_\infty(\kappa)\to \wap(G/N)/\B(G/N).\] Lemma \ref{quotquot} then provides the desired copy of
$\ell_\infty(\kappa) $ in $\wap(G)/\B(G)$.
\end{proof}

Theorem \ref{theorem:inquot} leads naturally also to an improvement  of \cite[Theorem 4.6]{chou82}.

\begin{theorem} \label{cor:nilquot} Let G be a
 a non-compact, locally compact, nilpotent group and put
$\kappa=\kappa(G)$.  Then the quo\-tient
$\wap(G)/\B(G)$ contains a linear isometric copy of
$\ell_\infty(\kappa).$
\end{theorem}

\begin{proof} The case $\kappa(Z(G))=\kappa(G)$ is already proved in Theorem \ref{center}. So we may assume that $\kappa(Z(G))<\kappa(G)$.
We argue by induction on the length $n$ of the  upper central series of $G$ (the nilpotency length of $G$)
\[\{e\}=G_0\subset G_{1}\subset\ldots \subset G_{n-1}\subset G_n=G\;\text{ with}\;
Z_{i+1}(G)/Z_i(G)=Z(G/Z_i(G)).\]

If $n=1$, then $G$ is Abelian and so Theorem \ref{center} or Theorem \ref{theorem:inquot} applies.

Assume as inductive hypothesis that the claim holds for groups of nilpotency length at most $n-1$ and suppose $G$ has nilpotency length $n$. Since $\kappa(G)=\kappa(Z(G))+\kappa(G/Z(G))$ and the case $\kappa(Z(G))=\kappa(G)$ has already been ruled out, we can assume that $\kappa(G/Z(G))=\kappa(G)$.
    Our inductive hypothesis ($\kappa(G/Z(G))$ has nilpotency length $n-1$) and Lemma \ref{quotquot} then  provide  the desired isometry.
\end{proof}

When $G$ is
an $IN$-group or a nilpotent group, we recover and improve further the results obtained in Section 4.

\begin{corollary}
   Let $G$ be a non-compact $IN$-group or a nilpotent group and let $\kappa$ be the compact covering of $G$.
 Then each of the quo\-tient spaces $\wap(G)/(\ap(G)\oplus C_0(G))$ and $\wap_0(G)/\B_0(G)$
contains a linear isometric copy of $\ell_\infty(\kappa).$
\end{corollary}

\begin{proof}
 That the first quotient contains a copy of $\ell_\infty(\kappa)$ follows directly from
 Theorem \ref{theorem:inquot} and Theorem \ref{cor:nilquot} if we recall the inclusion $\ap(G)\oplus C_0(G)\subseteq
\B(G)$ (see \cite[page 143]{chou82}).

For the second quotient, we argue as in Theorem \ref{theorem:inquot}.
None of the sets $T_\eta$, $\eta<\kappa,$ constructed in all cases needed in the proof of Theorem \ref{theorem:inquot},
is a $\B_0(G)$-interpolation set. On the other hand, proceeding precisely  as in
Theorems \ref{theorem:inquot} and  \ref{cor:nilquot} (and    using  the right  statements of Lemma \ref{quotquot}), we see that  $T=\cup_{\eta<\kappa}T_\eta$ is an approximable $\wap_0(G)$-interpolation set.
\end{proof}

\section{On the quotient of $\CB(G)$ by $\luc(G)$}

When $G$ is non-compact, non-discrete, locally compact group, Dzinotyiweyi showed in \cite{Dz} that the quotient $\CB(G)/\luc(G)$ is non-separable. When  $G$ is a non-precompact, topological group
which is not a P-group, this theorem was generalized and improved in  \cite[Theorem 3.1]{BF1} and \cite[Theorem 4.1]{BF2},
where a linear isometric copy of $\ell_\infty$ was proved to be contained in
$\CB(G)/\luc(G)$. This section is concerned again with locally compact groups. Our theorem is then more precise and definite. We prove,
there is a linear isometric copy of $\ell_\infty(\kappa)$ in $\CB(G)/\luc(G)$, where as before $\kappa$ is
the compact covering $G,$
if and only if $G$ is neither compact nor discrete.

\begin{lemma}\label{lem:seq}
  Every non-discrete locally compact group contains
a faithfully indexed sequence $\{x_n\colon n\in \N\}$ that converges
to the identity.
\end{lemma}

\begin{proof}
  A locally compact group always contains a compact subgroup $K$
such that $G/K$ is a metrizable topological space (see \cite[Theorem
4.3.29]{arhatkac}, for instance). Infinite compact groups on the
other hand always contain non-trivial convergent sequences
(\cite[Theorem 4.1.7 and Exercise 4.1.f]{arhatkac}). If $K$ is
infinite we are done. If $K$ is finite, $G$ is non-discrete and
metrizable, it therefore contains non-trivial convergent sequences.
\end{proof}

\begin{theorem}\label{Granirer}
 Let $G$ be a locally compact group.
Then $\CB(G)/\luc(G)$ contains a linear isometric copy of
$\ell_\infty(\kappa(G))$ if and only if $G$ is neither compact nor
discrete.
  \end{theorem}

\begin{proof}
The necessity is clear since $\CB(G)=\luc(G)$ if $G$ is either compact or discrete.

If $G$ is not compact we can find  a compact neighbourhood of the
identity $U$ and a $U^2$-right uniformly discrete subset
$X=\{x_\alpha\colon \alpha <\kappa\}\subseteq G$ with
$\kappa=\kappa(G)$. This is clear if $G$ is $\sigma$-compact. If
 $\kappa>\omega$, we consider
  $H=\langle U\rangle$,  the subgroup
  generated by $U$. Then $\kappa=|G:H|$ and any system of representatives of right cosets of $H$ constitutes
   an $H$-right uniformly discrete set of cardinality
  $\kappa$.

  Partition $X$ in $\kappa$-many countable subsets
$X=\bigcup_{\alpha<\kappa} X_\alpha$. Enumerate, for each $\alpha<\kappa,$
$X_\alpha=\{x_{\alpha,n}\colon n<\omega\}$. Since $G$ is not
discrete, $U$ contains (by Lemma \ref{lem:seq}) a faithfully indexed
sequence
 $S=\{s_n \colon n<\omega
\}$ converging to the identity. With these ingredients, we
define \[T_{\alpha,n}=\{s_jx_{\alpha,n}\colon 1\leq j\leq n\},\;
T_\alpha=\bigcup_n T_{\alpha,n}\;\text{ and }\;T=\bigcup_\alpha T_\alpha.\]
Obviously, $UT_\alpha\cap UT_{\alpha^\prime}=\emptyset$ for every $\alpha\neq \alpha^\prime<\kappa$.

Each set $T_\alpha$ fails to be an $\luc (G)$-interpolation
set. Indeed, the function $f\colon T_\alpha \to \C$ such that
\begin{align*}
 f(s_{2j}x_{\alpha,n})&=1\quad\text{ for every}\quad j,n\in \N \quad\text{with}\quad 1\leq 2j\leq
n\text{ and}\\f(s_{2j+1}x_{\alpha,n})&=-1  \quad\text{for every} \quad j,n\in \N \quad\text{ with}\quad
1\leq 2j+1\leq n
\end{align*}
cannot coincide on $T_\alpha$ with any
$\phi \in \luc(G)$, since given
 $\varepsilon>0$, we can choose $j$ large enough and $n\geq 2j+1$ so that
\[|\phi(s_{2j}x_{\alpha,n})-\phi(s_{2j+1}x_{\alpha,n})|<\varepsilon\quad\text{while}\quad
f(s_{2j}x_{\alpha,n})-f(s_{2j+1}x_{\alpha,n})=2.\]

We now prove that $T$ is an approximable $\CB(G)$-interpolation set.
Since the sequence $(s_j)$ is taken in $U$ and $X$ is right $U^2$-uniformly discrete, we see
that the open set $Ux_{\alpha,n}$ of $G$ contains no point from $T$ other than $s_jx_{\alpha,n},$ for $1\le j\le n$. Thus, $T$ is discrete.

Next we check that $T$ is closed. Let $x\notin T$. If for some $\alpha<\kappa,$ $n<\omega,$ and $1\le j\le n$, we have
$s_jx_{\alpha,n}\in Ux,$ then $x_{\alpha,n}\in s_j^{-1}Ux\subseteq
U^2x$. Note also that $s_jx_{\alpha,n}$ may be in $Ux$ for at most one $\alpha$ since  $UT_\alpha\cap UT_{\alpha^\prime}=\emptyset$ for every $\alpha\neq \alpha^\prime<\kappa$.
Thus,   \[Ux\cap T\subseteq\{s_j x_{\alpha, n}\in T \colon
x_{\alpha,n}\in U^2x,\quad\alpha<\kappa\;n<\omega,\; 1\le j\le n\}.\]
Since $X$ is right uniformly discrete and $U^ 2x$
is relatively compact,  the set \[\{x_{\alpha,n}\in U^2x\colon\quad\alpha<\kappa\;n<\omega,\; 1\le j\le n\}=X\cap U^2x\] must be finite. Therefore, there is
$k$ such that  \[Ux\cap T\subseteq \{s_j x_{\alpha,n_i}
\colon 1\leq j\leq n_i,\; i=1,\ldots,k\}.\] We conclude
that $Ux\cap T$ is finite, and so $T$ is closed. Since the
topological space underlying $G$ is  normal, $T$ is  an approximable
$\CB(G)$-interpolation set by Lemma \ref{kappa}.

Corollary   \ref{cor:main:constr} now implies that $\CB(G)/\luc(G)$ contains a linear
 isometric copy of $\ell_\infty(\kappa)$ with
 $\kappa=|X|=\kappa(G)$.
 \end{proof}

The equivalence of the first two statements of the following Corollary were proved by Baker and
Butcher in \cite{BB}, see also \cite{FV} for a different proof.
\begin{corollary} Let $G$ be a locally compact group with a compact covering number $\kappa$. Then the following statements are equivalent.
\begin{itemize}
 \item[(1)] $G$ is neither compact nor discrete.
 \item[(2)] $\CB(G)\ne \luc(G)$.
 \item[(3)] $\CB(G)/\luc(G)$ contains a linear isometric copy of $\ell_\infty(\kappa(G))$.
 \end{itemize}
\end{corollary}

\begin{proof}
$(1)\Longrightarrow (3)$ is proved in the theorem above.
$(3)\Longrightarrow (2)$ is obvious and $(2)\Longrightarrow (1)$ is   clear.
\end{proof}

\remark Theorem \ref{Granirer} implies a fortiori that the space $L^\infty(G)/\luc(G)$ as well as $L^\infty(G)/\wap(G)$ contains
a linear isometric copy of $\ell_\infty(\kappa).$
The arguments used in \cite[Section 4]{BF1}, may be applied again to deduce that the group algebra $L^1(G)$ is extremely
non-Arens regular whenever $\kappa$ is greater or equal to the local weight $w(G)$ of $G$ (this is the least cardinality of an open base at the identity of $G$.)
To obtain the full result, however,  harder work is necessary.  This  is achieved in our recent article \cite{FGenar}.
\endremark

\bigskip

\noindent {\sc Acknowledgement.} This paper was written when the
first author was visiting University of Jaume I in Castell\'on in
December 2010-January 2011. He would like to thank Jorge Galindo for
his hospitality and all the folks at the department of mathematics
in Castell\'on. The work was partially supported by Grant
INV-2010-20 of the 2010 Program for Visiting Researchers of
University Jaume I. This support is also gratefully acknowledged.

\medskip

\bibliographystyle{plain}
%\bibliography{d:/LOCALTEX/BIBLIOGRAFIAS/bibrep,d:/LOCALTEX/BIBLIOGRAFIAS/bibbohr,d:/LOCALTEX/BIBLIOGRAFIAS/bibsidon}
%\bibliography{D:/DOCS/BIBLIOGRAFIAS/bibrep,D:/docs/BIBLIOGRAFIAS/bibsidon,D:/docs/BIBLIOGRAFIAS/bibbohr}
%\end{document}

%\bibliography{C:/USERS/ADMIN/documents/BIBLIOGRAFIAS/bibrep,C:/USERS/ADMIN/documents/BIBLIOGRAFIAS/bibsidon,C:/USERS/ADMIN/documents/BIBLIOGRAFIAS/bibbohr}

\begin{thebibliography}{99}


\bibitem{OTA} {  O. T. Alas},\newblock {\em  Topological groups and uniform continuity}, \newblock  Portugal. Math. 30 (1971) 137--143.


\bibitem{arhatkac} {A.\,V.~Arhangel'skii and M.\,G.~Tkachenko},
\newblock {\em \it Topological Groups and Related Structures},
\newblock {\em Atlantis Series in Mathematics, vol.~I,
Atlantis Press/World Scientific, Paris--Amsterdam} (2008).


\bibitem{BB}
{J.~W. Baker and R.~J. Butcher.}
\newblock {\em The {Stone}{-}\mbox{{{\v C}ech}} compactification of a topological
 semigroup.}
\newblock Math. Proc. Camb. Phil. Soc,  80 (1976) 103--107.




\bibitem{BF} {  J. W. Baker and M. Filali,} {\em On the analogue of Veech's theorem in the
WAP-compactification of a locally compact group,} Semigroup Forum 65 no. 1
(2002) 107--112.



\bibitem{BJM} {J. F. Berglund, H. D. Junghenn \and P. Milnes,}
 {\em Analysis on Semigroups: Function Spaces, Compactifications,
Representations,}  Wiley, New York (1989).



\bibitem{BF1} {
A.~Bouziad and M.~Filali},
\newblock {\em On the size of quotients of function spaces on a topological group,} Studia Math.
202 (2011) 243--259.




\bibitem{BF2} {
A.~Bouziad and M.~Filali},
\newblock {\em The Stone-\v Cech compactification of a topological group  as a semigroup and the $SIN$ property,} Houston J.  Math.
 38 no. 4 (2012) 1329–1341.
to appear.


\bibitem{BIP} {T. Budak, N. Isik \and J. Pym,} {\em Minimal determinants of topological centres for some algebras associated with locally compact groups,} Bull. Lond. Math. Soc.  43 no. 3 (2011) 495--506.




\bibitem{Bu} {R. B. Burckel,}
{\em Weakly almost periodic functions on semigroups,}
 Gordon and Breach Science Publishers, New York-London-Paris 1970


 \bibitem{chou75} {C. Chou,} {\em Weakly almost periodic functions and almost convergent
functions on a group,} Trans. Amer. Math. Soc.
 206 (1975) 175--200.

\bibitem{chou79} {C. Chou,} {\em Uniform closures of Fourier-Stieltjes algebras,}
Proc. Amer. Math. Soc. 77 (1979) 99--102.


\bibitem{chou80}
{C. Chou,}
\newblock {\em Minimally  almost periodic groups.}
\newblock {\em J. Funct. Anal.} 36 (1980) 1--17.

 \bibitem{chou82} {C. Chou,} {\em Weakly almost periodic functions and Fourier-Stieltjes algebras of
locally compact groups}, Trans. Amer. Math. Soc., 274 no. 1 (1982)
141--157.

\bibitem{chou90}
{C. Chou,}
 {\em Weakly almost periodic functions and thin sets in discrete groups.}
 {\em Trans. Amer. Math. Soc.} 321 no. 1 (1990) 333--346.


\bibitem{CY} {P. Civin and B. Yood,} \emph{ The second conjugate space of a
Banach algebra as an algebra, Pacific J. Math. 11} (1961),   847--870.

\bibitem{CR} {W. Comfort and K. Ross}, {\em Pseudocompactness and uniform continuity in topological groups},
  Pacific J. Math.  16  (1966) 483--496.


\bibitem{DL} {H. G. Dales \and, A. T.-M. Lau,} {\em The second duals of Beurling algebras,}  Mem. Amer. Math. Soc.  177  (2005).

\bibitem{drury70}
S. W.  Drury.
\newblock {\em Sur les ensembles de Sidon.}
\newblock {\em C. R. Acad. Sci. Paris S\'er. A-B} 271 (1970) A162--A163.

\bibitem{dunklrami}
{C. F. Dunkl and D. E. Ramirez.}
\newblock {\em \it Topics in harmonic analysis}.
\newblock {\em Appleton-Century-Crofts [Meredith Corporation], New York, (1971).
\newblock Appleton-Century Mathematics Series.}

\bibitem{Dz} {A. M. H. Dzinotyiweyi,} {\em Nonseparability of quotient spaces of
function algebras on topological semigroups,} Trans. Amer. Math.
Soc. 272 (1982) 223--235.

\bibitem{enge77}
{R. Engelking.}
\newblock {\em \it General topology}.
\newblock {\em PWN---Polish Scientific Publishers, Warsaw,} (1977).



\bibitem{Ey} { P. Eymard,} {\em L'alg\`ebre de Fourier d'un groupe locallement
compact,}  Bull. Soc. Math. France 92 (1964) 181--236.



  \bibitem{F3} {M. Filali,} {\em On the actions of a
  locally compact group on some of its semigroup compactifications,}
  Math. Proc. Cambridge Philos. Soc. 143 (2007) 25--39.


\bibitem{FG} {M. Filali \and J. Galindo,} {\em
Approximable $\wap$- and $\luc(G)$-interpolation sets,}  Advances in Math. 233 (2013) 87--114.

\bibitem{FGenar} {M. Filali \and J. Galindo,} {\em
Extreme non-Arens regularity of the group
algebra,}  Preprint  (2013).

\bibitem{FP} {M. Filali \and J. S. Pym,} {\em Right cancellation in
 the $LUC$-compactification of a locally compact group,}
  Bull. London Math. Soc.  35  (2003) 128--134.

\bibitem{FS2}
{M. Filali \and P. Salmi},
  {\em Slowly oscillating functions in semigroup compactification
  and convolution algebras,} Journal of Functional Analysis  250 no. 1 (2007) 144--166.



\bibitem{FV} {M. Filali \and T. Vedenjuoksu,}
 {\em The Stone-\v Cech compactification of a topological group and the $\beta-$extension property,}
Houston J. Math.  36 no. 2 (2010) 477-488.




\bibitem{gali10}
{J. Galindo},
\newblock {\em On group and semigroup compactifications of topological groups.}
\newblock Manuscript in preparation.



%\bibitem{gilljeri}
%Leonard Gillman and Meyer Jerison.
%\newblock {\em Rings of continuous functions}.
%\newblock The University Series in Higher Mathematics. D. Van Nostrand Co.,
%  Inc., Princeton, N.J.-Toronto-London-New York, 1960.


\bibitem{GH} {
C. C. Graham \and K. E. Hare,}  {\em Interpolation and Sidon sets for compact groups,} CMS Books in Mathematics/Ouvrages de Math\'ematiques de la SMC. Springer, New York, (2013).

\bibitem{G2} {E. E. Granirer}, {\em Exposed points of convex sets and weak sequential convergence,}
 Memoirs of the American Mathematical Society, No. 123. {\em American
Mathematical Society, Providence, R.I.,} (1972).



\bibitem{Gr} {E. E. Granirer,} \emph{
The radical of $L^{\infty}(G)^*$,}
Proc. Amer. Math. Soc. 41  (1973)  321--324.



\bibitem{grossmosk71} { S. Grosser, \and M.  Moskowitz}, \emph{ Compactness conditions in topological groups},
 J. Reine Angew. Math.  246  (1971) 1--40.
		



\bibitem{G} {S. L. Gulick,} \emph{
Commutativity and ideals in the biduals of topological algebras,}
Pacific J. Math. 18  (1966), 121--137.

\bibitem{HR} {E. Hewitt \and K. A. Ross,} {\em Abstract harmonic analysis I,}
Springer-Verlag, Berlin, (1963).


\bibitem{HS} { N. Hindman and D. Strauss}, {\em Algebra in
the Stone-\v Cech Compactification,} de Gruyter Exp. Math.  27,
Walter de Gruyter, Berlin, (1998).


\bibitem{jech} {T. Jech}, {\em \it Set Theory,} Springer,
{\em    The third millennium edition, revised and expanded, Berlin}
(2006).

\bibitem{L} {A. T. -M. Lau}, \emph{ Continuity of Arens multiplication on the dual
space of bounded uniformly continuous functions on locally compact groups
 and topological semigroups,} Math. Proc. Cambridge Philos. Soc. 99 (1986)  273--283.



\bibitem{LL} {A. T. M. Lau and V. Losert,} \emph{ On the second conjugate
 of $L^1(G)$ of a locally compact group,} J. London Math. Soc. 37 (1988) 464--470.



 \bibitem{LP} {A. T. M. Lau, J. S. Pym,} {\em The topological centre of
  a compactification of a locally compact group,} Math. Z. 219 no. 4 (1995) 567--579.


  \bibitem{LU} {A. T. M. Lau, A. \"Ulger,} {\em Topological centers of
  certain dual algebras,} Trans. Amer. Math. Soc. 348 (1996) 1191--1212.

\bibitem{deL1} K. de Leeuw and I. Glicksberg, {\em Almost periodic functions on semigroups,} Acta Math. 105 (1961) 99-140.

\bibitem{deL2}   { K. de Leeuw and I. Glicksberg,} {\em The decomposition of certain group representations}
Analyse Math. 15 (1965) 135-192.



  \bibitem{N} {M. Neufang,} {\em A unified approach to the topological
   centre problem for certain Banach algebras arising in abstract
   harmonic analysis,} Arch. Math. (Basel) 82 no. 2 (2004) 164--171.


\bibitem{neum54}
{B.~H. Neumann},
\newblock {\em Groups covered by finitely many cosets.}
\newblock {\em Publ. Math. Debrecen} (1954) 227--242 (1955).





\bibitem{P}{   J. Pym}, \emph{A note on $G^{\mathcal{ LUC}}$ and Veech's theorem},
Semigroup Forum 59 (1999) 171--174.


   \bibitem{ra} {D. E. Ramirez,} {\em Weakly almost periodic functions and Fourier-Stieltjes transforms,} Proc. Amer. Math. Soc. 19 (1968) 1087--1088.



\bibitem{rose70} {H. P. Rosenthal}, \emph{
On injective Banach spaces and the spaces $L^\infty(\mu)$
for finite measures $\mu$}, Acta Math.  124 (1970) 205--248.



    \bibitem{ru} {W. Rudin,}{\em Weak almost periodic functions and
    Fourier-Stieltjes transforms,} Duke Math. J. 26 (1959) 215--220.



  \bibitem{rup} {W. Ruppert,}{\em On weakly almost periodic sets,}
   Semigroup Forum 32 (1985) 267--281.


 \bibitem{rup2} {W. Ruppert,} {\em Compact semitopological
  semigroups: an intrinsic theory,} Lecture Notes in Mathematics,
  1079. Springer-Verlag, Berlin, 1984.


\bibitem{Ve}{  W. A. Veech.} Topological dynamics.\textit{ Bull. Amer.
Math. Soc.} 83 (1977) 775--830.



    \bibitem{V} {W. A. Veech,} {\em Weakly almost periodic functions on semisimple
     Lie groups,} Monatsh. Math. 88 no. 1 (1979)  55--68.



 \bibitem{Y} {N. J. Young,} {\em The irregularity of multiplication in
  group algebras,} Quart
J. Math. Oxford Ser. 24 (1973) 59--62.
\end{thebibliography}
%\end{document}

\end{document}